\documentclass[reqno]{amsart}
\usepackage[usenames,dvipsnames]{xcolor}
\usepackage{amsfonts, fancyhdr, amsmath, amssymb, hyperref, url, amsthm, subfigure, tikz-cd, verbatim, enumitem,lscape, todonotes, varioref}
\usetikzlibrary{decorations.pathmorphing}
\usepackage[mathscr]{eucal} 
\usepackage[noabbrev, capitalise, nameinlink]{cleveref}

\definecolor{crimson}{rgb}{0.86, 0.08, 0.24}
\definecolor{darkcyan}{rgb}{0.0, 0.55, 0.55}

\hypersetup{
colorlinks={true},
linkcolor= {darkcyan},
citecolor= darkcyan}

\numberwithin{equation}{section}

\newtheorem{thm}{Theorem}[section]
\newtheorem{lem}[thm]{Lemma}

\newtheorem{prop}[thm]{Proposition}
\newtheorem{df}[thm]{Definition}
\newtheorem*{df*}{Definition}
\newtheorem{conj}[thm]{Conjecture}
\newtheorem{cor}[thm]{Corollary}
\newtheorem{rmk}[thm]{Remark}
\newtheorem{construction}[thm]{Construction}
\newtheorem{exam}[thm]{Example}
\newtheorem{quest}[thm]{Question}

\newtheorem{thmx}{Theorem}

\newcommand{\EF}{\widetilde{E}\mathcal{F}}
\newcommand{\HF}{H\mathbb{F}_2}
\newcommand{\MUR}{MU_\mathbb{R}}

\newcommand{\MUCnplusone}{MU^{(\!(C_{2^{n+1}})\!)}}
\newcommand{\MUCn}{MU^{(\!(C_{2^n})\!)}}
\newcommand{\MUG}{MU^{(\!(G)\!)}}
\newcommand{\BPR}{BP_\mathbb{R}}
\newcommand{\BPCfour}{BP^{(\!(C_{4})\!)}}
\newcommand{\BPCtwoCtwo}{BP^{(\!(C_2 \times C_2)\!)}}
\newcommand{\BPCeight}{BP^{(\!(C_{8})\!)}}
\newcommand{\BPQeight}{BP^{(\!(Q_{8})\!)}}
\newcommand{\BPCfourone}{BP^{(\!(C_{4})\!)}\langle 1 \rangle}
\newcommand{\BPCn}{BP^{(\!(C_{2^n})\!)}}
\newcommand{\BPCk}{BP^{(\!(C_{2^k})\!)}}
\newcommand{\BPCkplusone}{BP^{(\!(C_{2^{k+1}})\!)}}
\newcommand{\BPG}{BP^{(\!(G)\!)}}
\newcommand{\BPGmodtwo}{BP^{(\!(G/C_2)\!)}}
\newcommand{\BPGmodCk}{BP^{(\!(G/C_{2^k})\!)}}
\newcommand{\BPCnminusone}{BP^{(\!(C_{2^{n-1}})\!)}}
\newcommand{\BPCnplusone}{BP^{(\!(C_{2^{n+1}})\!)}}
\newcommand{\BPCnminuskplusone}{BP^{(\!(C_{2^{n-k+1}})\!)}}

\newcommand{\tee}{\bar{t}}
\newcommand{\ti}{\bar{t}_i}
\newcommand{\CP}{\mathbf{CP}^\infty}
\newcommand{\RP}{\mathbf{RP}^\infty}

\newcommand{\smashover}[1]{{\underset{#1}{\wedge}}}

\newcommand{\Pb}{\mathcal{P}^*}
\newcommand{\Mpi}{{\underline{\pi}}}

\newcommand{\Cn}{C_{2^n}}
\newcommand{\Cnplusone}{C_{2^{n+1}}}
\newcommand{\Cnminusone}{C_{2^{n-1}}}
\newcommand{\Ck}{C_{2^k}}
\newcommand{\Ckplusone}{C_{2^{k+1}}}
\newcommand{\Cnminuskplusone}{C_{2^{n-k+1}}}

\newcommand{\FF}{\mathcal{F}}
\newcommand{\Z}{\underline{\mathbb{Z}}}
\newcommand{\D}{\mathscr{D}}
\newcommand{\N}{\mathbb{N}}
\newcommand{\E}{\underline{\mathcal{E}}}
\newcommand{\sm}{\wedge}
\newcommand{\rhobar}{\bar{\rho}}

\newcommand{\Aut}{\textup{Aut}}

\newcommand{\im}{\textup{im}}
\newcommand{\id}{\textup{id}}

\newcommand{\SliceSS}{\textup{SliceSS}}

\newcommand{\Sp}{\textup{Sp}}

\title[Transchromatic phenomena in the slice spectral sequence]{Transchromatic phenomena in the equivariant slice spectral sequence}

\author{Lennart Meier}
\address{Mathematical Institute, Utrecht University, Utrecht, 3584 CD, the Netherlands}
\email{f.l.m.meier@uu.nl}

\author{XiaoLin Danny Shi}
\address{Department of Mathematics, 
University of Washington, 
4110 E Stevens Way NE,
Seattle, WA 98195
}
\email{dannyshixl@gmail.com}

\author{Mingcong Zeng}
\address{Department of Mathematics, 
Ny Munkegade 118, 
8000 Aarhus C, 
Denmark}
\email{mingcongzeng@gmail.com}

\begin{document}

\maketitle
\begin{abstract}
In this paper, we prove a transchromatic phenomenon for Hill--Hopkins--Ravenel and Lubin--Tate theories. This establishes a direct relationship between the equivariant slice spectral sequences of height-$h$ and height-$(h/2)$ theories. As applications of this transchromatic phenomenon, we prove periodicity and vanishing line results for these theories.
\end{abstract}

{\hypersetup{linkcolor=black}
\tableofcontents}

\section{Introduction}

\subsection{Motivation and main theorem} \label{subsec:Motivation}

The equivariant slice filtration was developed by Hill, Hopkins, and Ravenel in their celebrated solution of the Kervaire invariant one problem \cite{HHR}.  This filtration gives rise to the equivariant slice spectral sequence, which is a valuable computational tool in equivariant homotopy theory for analyzing the norms of the Real bordism spectrum $\MUG := N_{C_2}^G (\MUR)$.

Let $E_h = E(k, \Gamma_h)$ be the Lubin--Tate spectrum associated to a formal group law $\Gamma_h$ of height $h$ over a finite field $k$ of characteristic 2.  Goerss, Hopkins, Miller, and Lurie have shown that $E_h$ is a commutative ring spectrum, and there is an action of $\Aut(\Gamma_h)$ on $E_h$ by commutative ring maps \cite{HopkinsMiller, GoerssHopkins, LurieElliptic2}.  For any finite group $G$ of $\Aut(\Gamma_h)$, we can view $E_h$ as a $G$-equivariant commutative ring spectrum using the cofree functor $F(EG_+, -)$.

The theories $E_h^{hG}$ hold significant importance in chromatic homotopy theory, as they are periodic theories capable of detecting periodic patterns and nontrivial families of elements in the stable homotopy groups of spheres \cite{AdamsJ4, RavenelOddKervaire,HHR, HurewiczImages, BehrensMahowaldQuigleyTMF}.  They are also used in the study of the $K(h)$-local sphere \cite{GoerssHennMahowald, GoerssHennMahowaldRezk, BehrensModular, Henn2007, HennKaramanovMahowald, Lader, GoerssHennMahowaldRational, BobkovaGoerss, BeaudryGoerssHenn}.  At height 1, there is a well-established understanding of these theories due to their relationship with complex and real $K$-theory \cite{Adams1974, BousfieldLocalization, RavenelLocalization, Bousfield1985}.  At height 2, they are closely related to topological modular forms with level structures \cite{HopkinsMahowald, BauerTMF, MahowaldRezk, tmfbook, BehrensOrmsby, HillMeier, HHRKH}.

In this paper, we focus our attention at the prime 2.  Prior to \cite{HahnShi, HSWX}, computations of $E_h^{hG}$ at the prime 2 were limited to heights $h \leq 2$.  Explicitly computing these theories posed challenges, partly due to the lack of geometric insight into the $G$-action on $E_h$.

In \cite{HahnShi}, Hahn and the second author produced a $G$-equivariant orientation
\[\MUG \longrightarrow E_h.\]
This orientation establishes a connection between the obstruction theoretic $G$-action on $E_h$, induced from $\Aut(\Gamma)$, to the geometric $G$-action on $\MUG$, which arises from complex conjugation.  This connection enables us to construct $E_h$ as a localization of a quotient of $\MUG$.

For $G = \Cnplusone$ and $h = 2^n \cdot m$, Beaudry--Hill--Shi--Zeng \cite{BHSZ} constructed $G$-equivariant models of $E_h$ using quotients of the norm $\BPG := N_{C_2}^{G}(\BPR)$, where $\BPR$ is the Real Brown--Peterson spectrum.  These quotients, denoted as $\BPG \langle m \rangle$, are the equivariant analogues of the classical truncated Brown--Peterson spectrum $BP\langle m \rangle$.  Consequently, computing the fixed points $E_h^{hG}$ is directly tied to computing the $G$-fixed points of $\BPG \langle m \rangle$.

Using the equivariant slice spectral sequence, the following computations have been made: 
\begin{enumerate}
\item The $C_2$-fixed points of $\BPR \langle m \rangle$ for all $m \geq 1$.  This yields $E_m^{hC_2}$ for all $m \geq 1$ (\cite{HuKriz}, \cite{HahnShi});
\item The $C_4$-fixed points of $\BPCfour \langle 1 \rangle$.  This yields $E_2^{hC_4}$ (\cite{BehrensOrmsby}, \cite{HHRKH});
\item The $C_4$-fixed points of $\BPCfour \langle 2 \rangle$.  This yields $E_4^{hC_4}$ (\cite{HSWX}).
\end{enumerate}

The goal of this paper is to achieve a systematic understanding of the structural properties of $E_h^{hG}$ across different groups and heights.  This includes differentials, periodicities, and vanishing lines.  A key conjecture by Hill--Hopkins--Ravenel states that the slice spectral sequences of $E_h^{hG}$ exhibit connections to each other as we vary $h$ and $G$:

\begin{conj}[Transchromatic Phenomenon] \label{conj:HHRConjecture}
The slice spectral sequences of higher height theories contain isomorphism regions to the slice spectral sequences of lower height theories.
\end{conj}

\cref{conj:HHRConjecture} originated from the work of Hill, Hopkins, and Ravenel during their study of $\BPCfourone$ \cite{HHRKH}, where they observed isomorphism regions between the $C_4$-slice spectral sequence of $\BPCfourone$ ($E_2^{hC_4}$) and the $C_2$-slice spectral sequence of $\BPR \langle 1 \rangle$ ($E_1^{hC_2}$).  A similar phenomenon was later discovered between the $C_4$-slice spectral sequence of $\BPCfour \langle 2 \rangle$ ($E_4^{hC_4}$) and the $C_2$-slice spectral sequence of $\BPR\langle 2 \rangle$ ($E_2^{hC_2}$) \cite{HSWX}. 

\vspace{0.1in}
To this end, the main theorem of this paper is the following: 

\begin{thmx}[Transchromatic Isomorphism]\label{thm:introThmTranschromaticIsom}
For $G = \Cnplusone$ and $m \geq 1$, there is a shearing isomorphism $d_{2r-1} \leftrightsquigarrow d_r$ between the following regions of spectral sequences: 
\begin{enumerate}
\item The $G$-slice spectral sequence of $\BPG \langle m \rangle$ on or above the line of slope $1$; and  
\item The $(G/C_2)$-slice spectral sequence of $\BPGmodtwo \langle m \rangle$.
\end{enumerate}
\end{thmx}

More precisely, \cref{thm:introThmTranschromaticIsom} states that all the differentials in the $G$-slice spectral sequence of $\BPG \langle m \rangle$ on or above the line of slope 1 have lengths of the form $(2r - 1)$, where $r \geq 2$. Additionally, the transchromatic isomorphism induces a one-to-one correspondence between these differentials and the $d_r$-differentials in the $(G/C_2)$-slice spectral sequence of $\BPGmodtwo \langle m \rangle$.

Combined with the results in \cite{BHSZ}, \cref{thm:introThmTranschromaticIsom} establishes an explicit correspondence between differentials in the $\Cnplusone$-slice spectral sequence of $E_h$ on or above the line of slope 1 and the differentials in the $\Cn$-slice spectral sequence of $E_{h/2}$. 

For any finite subgroup $G \subset \mathbb{G}_h$, the 2-Sylow subgroup $H$ of $G \cap \mathbb{S}_h$ is isomorphic to either a cyclic group of order a power of $2$ or the quaternion group $Q_8$ \cite{Hewett, Hewett2, Bujard}.  In light of this classification, \cref{thm:introThmTranschromaticIsom} resolves \cref{conj:HHRConjecture} for the case when the Sylow subgroup $H$ is cyclic.

Computationally, \cref{thm:introThmTranschromaticIsom} significantly simplifies the complexity of the slice spectral sequences of Hill--Hopkins--Ravenel ($\BPCnplusone \langle m \rangle$) and Lubin--Tate theories ($E_h^{h\Cnplusone}$).  More specifically, the Stratification Theorem \cite[Theorem~A]{MeierShiZengStratification} shows that the $\Cnplusone$-slice spectral sequence is stratified into $(n+1)$ regions, separated by the lines $\mathcal{L}_{2^i-1}$ of slope $(2^i -1)$, $0 \leq i \leq n+1$ (see \cref{fig:IntroPicThmATransIsom}).  \cref{thm:introThmTranschromaticIsom} shows that all the differentials in the $n$ regions above the line $\mathcal{L}_1$ are accounted for by the $\Cn$-slice spectral sequence of $\BPCn \langle m \rangle$ ($E_{h/2}^{h \Cn}$).  As a consequence, the new differentials needing computation are limited to the conical region between $\mathcal{L}_0$ and $\mathcal{L}_1$.  The vanishing lines results in \cref{subsec:PeriodicityandVanishingLines} will further narrow down the possibilities for differentials in this region.  

\begin{figure}
\begin{center}
\makebox[\textwidth]{\includegraphics[trim={0cm 0cm 0cm 0cm}, clip, scale = 0.6]{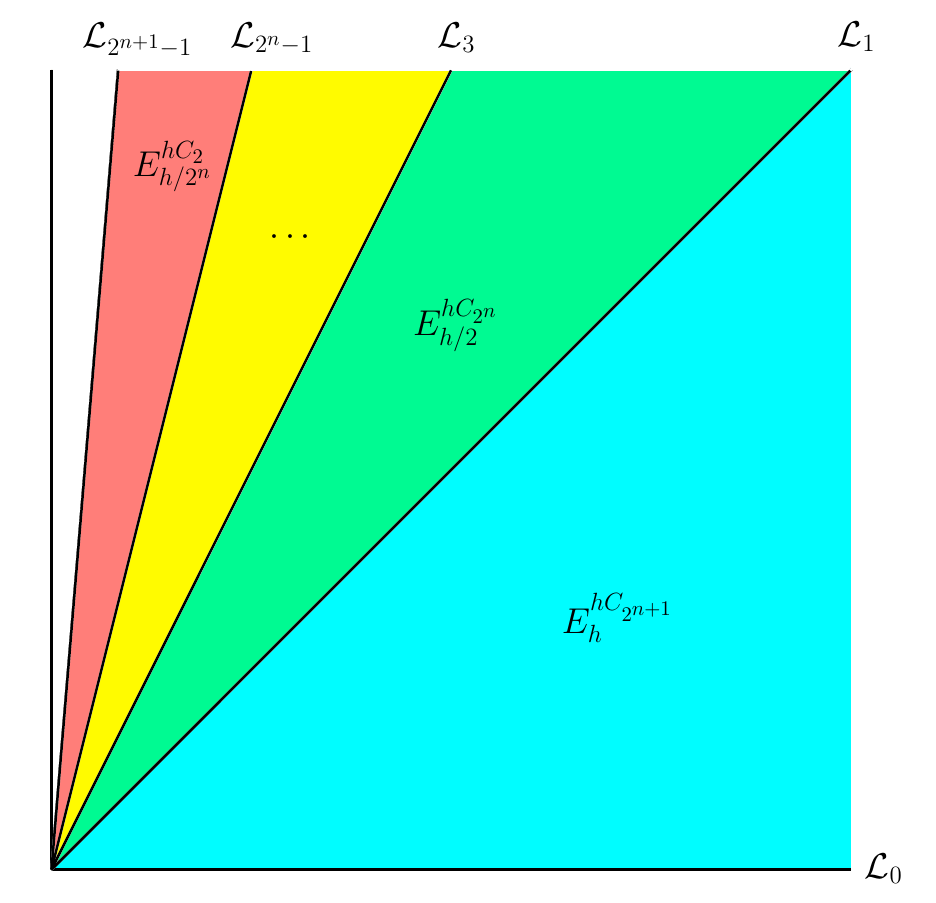}}
\caption{The Transchromatic Isomorphism Theorem.}
\hfill
\label{fig:IntroPicThmATransIsom}
\end{center}
\end{figure}

\subsection{Periodicities and vanishing lines}\label{subsec:PeriodicityandVanishingLines}

\cref{thm:introThmTranschromaticIsom} establishes explicit correspondence formulas between differentials in the corresponding isomorphism regions of the slice spectral sequences.  In addition to providing an inductive approach for computing differentials, it also offers a new method for studying $RO(G)$-graded periodicities and vanishing lines for the higher-height theories through the lower-height theories.

The classical 2-periodicity of $E_h$ can be refined to a $C_2$-equivariant $\rho_2$-periodicity, and, $G$-equivariantly, a $\rho_G$-periodicity \cite{HahnShi}.  On top of these periodicities, there are numerous other $RO(G)$-graded periodicities that are of the form $(|V| - V)$, where $V \in RO(G)$.  They correspond to the survival of orientation classes $u_V$ to the $\mathcal{E}_\infty$-page of the slice spectral sequence of $E_h$.

The $RO(G)$-graded periodicities of $E_h$ play a crucial role in the proof of the Kervaire invariant one problem, producing the 256-periodicity of the detecting spectrum $\Omega$ \cite[Theorem~1.7]{HHR}. They are also valuable for understanding periodic behaviors for differentials in the slice spectral sequence \cite{HHRKH, HSWX, BHLSZ}, as well as for computing Picard groups \cite{MathewStojanoskaPicard, HeardMathewStojanoska, HillMeier, BBHS, HeardLiShi}.

\begin{thmx}[Periodicity, \cref{thm:ROGPeriodicityTranschromatic}]\label{thm:introThmPeriodicity}
Let $G = \Cnplusone$, $h = 2^n \cdot m$, and suppose $V \in RO(G/C_2)$.  Then $(|V|-V)$ is an $RO(G)$-graded periodicity for $E_h$ if and only if $(|V|- V)$ is an $RO(G/C_2)$-graded periodicity for $E_{h/2}$.  Here, $V$ is treated as an element in $RO(G)$ through the pullback along the map $G \to G/C_2$.      
\end{thmx}

\cref{thm:introThmPeriodicity} provides an inductive method for establishing $RO(G)$-graded periodicities.  It shows that the periodicities obtained for the smaller height theories immediately extend to the higher height theories.  As we move up the height and group, the $RO(G)$-graded periodicities build up and accumulate, rather than being random.  

More specifically, when $G = \Cnplusone$ is the cyclic group of order $2^{n+1}$ and we are working 2-locally, the $RO(G)$-graded homotopy groups are graded by the irreducible real representations $\{1, \sigma, \lambda_1, \ldots, \lambda_n\}$, where 1 is the trivial representation, $\sigma$ is the sign representation, and $\lambda_i$ is the 2-dimensional real representation corresponding to rotation by $\left(\frac{\pi}{2^i}\right)$.  Similarly, when considering the quotient group $G/C_2$, the $RO(G/C_2)$-graded homotopy groups are graded by the representations $\{1, \sigma, \lambda_1, \ldots, \lambda_{n-1}\}$.  By applying \cref{thm:introThmPeriodicity} to the $RO(G/C_2)$-graded periodicities of $E_{h/2}$, we can immediately obtain all the $RO(G)$-graded periodicities of $E_h$ that are generated by the representations $\{1, \sigma, \lambda_1, \ldots, \lambda_{n-1}\}$.

Our second application of \cref{thm:introThmTranschromaticIsom} is in establishing vanishing lines.  Classically, it is a consequence of the Nilpotence Theorem of Devinatz, Hopkins, and Smith that the homotopy fixed point spectral sequences of $E_h$ all admit strong horizontal vanishing lines at \textit{some} finite filtration (see \cite[Section~5]{devinatzHopkins} and \cite[Section~2.3]{BeaudryGoerssHenn}).  This means that there exists a finite number $R > 0$ such that the spectral sequence collapses after the $\mathcal{E}_R$-page, with no elements of filtration greater than or equal to $R$ surviving to the $\mathcal{E}_\infty$-page.

While theoretically useful, the existence of the horizontal vanishing line cannot be used to compute any differentials, because there is no knowledge of exactly \textit{where} the vanishing occurs.  Knowing a bound (or even better, the exact location) for the vanishing line would significantly aid in computing differentials and in studying their structural properties.

In \cite{DuanLiShiVanishing}, Duan, Li, and the second author established a bound for the strong horizontal vanishing line in the $\Cnplusone$-homotopy fixed point spectral sequence and the $\Cnplusone$-slice spectral sequence of $E_{h}$.  

Combining \cref{thm:introThmTranschromaticIsom} with the results in \cite{DuanLiShiVanishing}, we prove the following result:

\begin{thmx}[Vanishing Lines, \cref{thm:VanishingLineGeneralSlope}] \label{thm:introThmVanishingLines}
Let $G = \Cnplusone$ and $h = 2^n \cdot m$.  For all $V \in RO(G)$ and $0 \leq k \leq n$, the line $\mathcal{L}_{2^k-1, N_k}$ defined by the equation 
\[s = (2^k-1)(t-s) + N_k\]
is a strong vanishing line in the slice spectral sequence of $E_h$ within the region $t-s \geq 0$.  Here, $N_k$ is the constant $(2^{h/2^k+n+1} - 2^{n+1}+2^k)$.  This strong vanishing line satisfies the following properties: 
\begin{enumerate}
\item There are no classes on or above $\mathcal{L}_{2^k-1, N_k}$ that survive to the $\mathcal{E}_\infty$-page; 
\item Differentials originating on or above the line $\mathcal{L}_{2^k-1}$ have lengths at most ${N_k - (2^k-1) = 2^{h/2^k+n+1} - 2^{n+1} + 1}$; 
\item All differentials originating below $\mathcal{L}_{2^k-1}$ have targets strictly below $\mathcal{L}_{2^k-1, N_k}$.
\end{enumerate}    
\end{thmx}

\begin{figure}
\begin{center}
\makebox[\textwidth]{\includegraphics[trim={0cm 0cm 0cm 0cm}, clip, scale = 0.6]{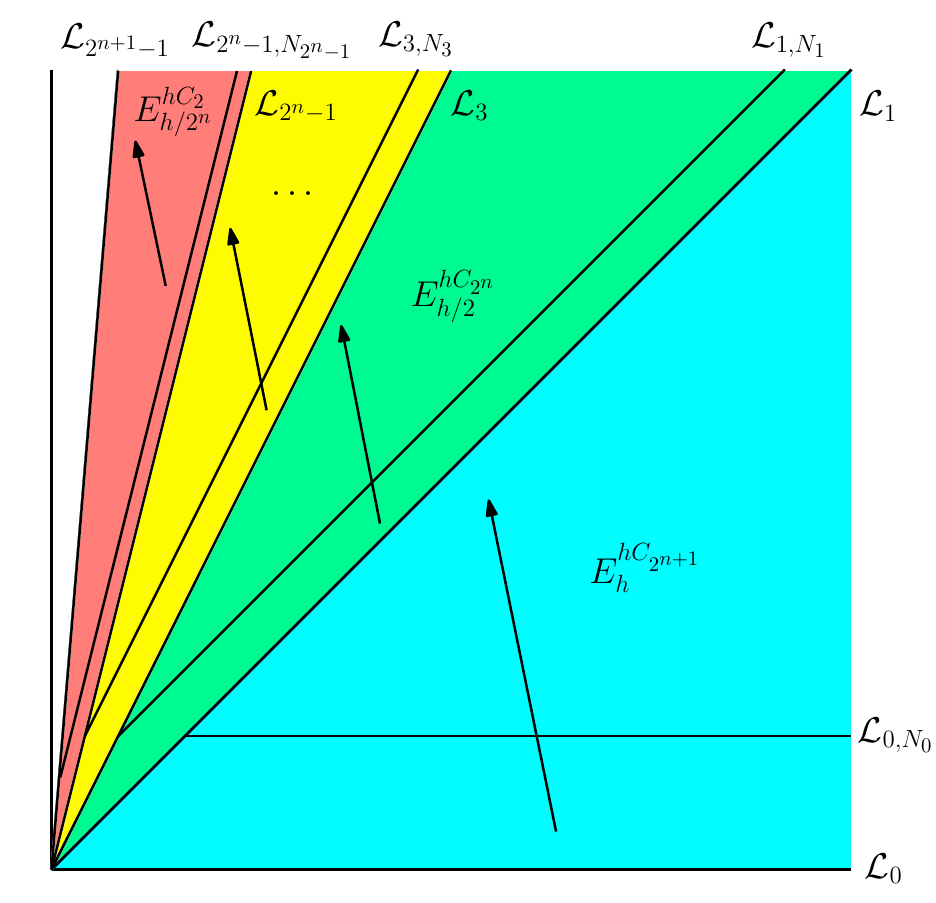}}
\caption{Vanishing Lines in the slice spectral sequence.}
\hfill
\label{fig:IntroPicThmCVanishingLines}
\end{center}
\end{figure}

The collection of strong vanishing lines established in \cref{thm:introThmVanishingLines} is unique to the equivariant slice spectral sequence.  With the exception of the horizontal vanishing line $\mathcal{L}_{0, N_0}$, the remaining vanishing lines do not appear in the homotopy fixed point spectral sequence of $E_h$.  This absence is due to the absence of stratification regions in the homotopy fixed point spectral sequence, which, in turn, prevents the observation of transchromatic phenomena.

In contrast, the slice spectral sequence has $(n+1)$ stratification regions \cite{MeierShiZengStratification}.   \cref{thm:introThmTranschromaticIsom} grants each stratification region with a vanishing line of its own.  Consequently, in addition to the original horizontal vanishing line established in \cite{DuanLiShiVanishing}, there are $n$ more vanishing lines of slopes $1$, $3$, $\ldots$, $(2^{n}-1)$, respectively.  

Together, these vanishing lines offer more precise bounds for both the \textit{lengths} and \textit{locations} of differentials, as illustrated in \cref{fig:IntroPicThmCVanishingLines}.  Rather than computing the entire spectral sequence all at once, one can focus on proving differentials in each of the $(n+1)$ conical region individually.  This significantly reduces the number of potential differential possibilities.  This feature makes the slice spectral sequence an effective computational tool with the potential to greatly expand our ability to analyze and compute $E_h^{hG}$ at higher heights.


\subsection{The stratification tower of Hill--Hopkins--Ravenel theories}\label{subsec:introStratificationTowerHHRTheories}

We will now outline our strategy for proving \cref{thm:introThmTranschromaticIsom}.  It is worth noting that, although \cref{conj:HHRConjecture} originally emerged from explicit spectral sequence computations, our proof of \cref{thm:introThmTranschromaticIsom} does not require the computation of explicit differentials.  Instead, we will prove a more general statement on the level of filtered objects, which establishes $RO(G)$-graded shearing isomorphisms for all generalized Hill--Hopkins--Ravenel theories $\BPCn \langle I \rangle$.

By the work of Hill--Hopkins--Ravenel \cite[Section~5]{HHR}, we have the following isomorphism: 
\[\pi_{*\rho_2}^{C_2} \BPCn \cong \mathbb{Z}_{(2)}[\Cn \cdot \tee_1^{\Cn}, \Cn \cdot \tee_2^{\Cn}, \ldots].\]
Here, $\tee_i^{\Cn} \in \pi_{(2^i-1)\rho_2}^{C_2} \BPCn$, and $\Cn \cdot \tee_i^{\Cn}$ denotes the set 
\[\left\{\tee_i^{\Cn}, \gamma \tee_i^{\Cn}, \ldots, \gamma^{2^{n-1}-1} \tee_i^{\Cn}\right\}\]
containing $2^{n-1}$ elements, where $\gamma$ represents the generator of $\Cn$.  The method of twisted monoid rings \cite[Section~2]{HHR} allows us to form quotients of $\BPCn$ by collections of permutation summands of the form $\Cn \cdot \tee_i^{\Cn}$.

\begin{df} \label{df:introDefinitionHHRTheory} \rm
For $I \subseteq \mathbb{N}$, the \textit{generalized Hill--Hopkins--Ravenel theory} $\BPCn \langle I \rangle$ is defined as
\[\BPCn \langle I \rangle := \BPG/(\Cn \cdot \tee_{j}^{\Cn}\,|\, j \notin I).\]
\end{df}

The theories $\BPCn \langle I \rangle$ are studied in \cite{BHLSZ}.  In particular, \cite[Theorem~A]{BHLSZ} shows that the slice associated graded of $\BPCn \langle I \rangle$ is the generalized Eilenberg--Mac Lane spectrum 
\[H\Z\left[\Cn \cdot \tee_{i}^{\Cn}\,|\, i \in I\right].\]

When $I = \{1, 2, \ldots, m\}$, the spectrum $\BPCn \langle I \rangle$ is the classical Hill--Hopkins--Ravenel theory 
\[\BPCn \langle m \rangle = \BPCn/(\Cn \cdot \tee_{m+1}^{\Cn}, \Cn \cdot \tee_{m+2}^{\Cn}, \ldots).\]

For a fixed $I \subseteq \N$, the key to proving \cref{thm:introThmTranschromaticIsom} lies in isolating and comparing regions in the slice spectral sequences associated with the theories $\BPCn \langle I \rangle$ as we vary $n$.  To achieve this, we will utilize the localized slice spectral sequence and the stratification tower, developed by the authors in \cite{MeierShiZengHF2} and \cite{MeierShiZengStratification}.

\begin{thmx}[Localized Slice Tower Isomorphism, \cref{thm:DualTowerEquivalenceMain}]\label{thm:introThm5DualTowerEquivalence}
Suppose $n$ is a positive integer.  For all $1 \leq k \leq n$, we have the following equivalence of towers:
\[\EF[C_{2^{k+1}}] \wedge P_\bullet\left(\BPCnplusone \langle I \rangle \right) \simeq \Pb_{C_{2^{n+1}}/C_{2^k}}\D^k\left(\EF[C_{2}] \wedge P_\bullet (\BPCnminuskplusone \langle I \rangle) \right).\]
Here, $P_\bullet(-)$ denotes the dual slice tower, $\Pb_{\Cnplusone/\Ck}(-)$ is the pullback functor as defined in \cref{df:PullBack}, and $\D(-)$ is the doubling operation as defined in \cref{df:doubleTower}.
\end{thmx}

\cref{thm:introThm5DualTowerEquivalence} establishes equivalences between various localized dual slice towers, leading to corresponding equivalences in the localized slice spectral sequences across different groups.  This equivalence results in a \textit{shearing isomorphism} of the corresponding localized slice spectral sequences.  This shearing isomorphism, combined with the Slice Recovery Theorem \cite[Theorem~3.3]{MeierShiZengStratification}, produces shearing isomorphisms for the $RO(G)$-graded slice spectral sequences of $\BPCnplusone \langle I \rangle$, $n \geq 0$.

\begin{thmx}[Shearing Isomorphism, \cref{thm:ShearingGeneralTheory}, \cref{thm:TranschromaticMain}]\label{thm:introThmTranschromaticIsomBPGI}
For all $V \in RO(C_{2^{n+1}})$ and $1 \leq k \leq n$, there is a shearing isomorphism $d_{2^kr - (2^k-1)} \leftrightsquigarrow d_r$ between the following regions of spectral sequences:
\begin{enumerate}
\item The $(V+t-s, s)$-graded page of ${\SliceSS(\BPCnplusone \langle I \rangle)}$ on or above the line $\mathcal{L}^V_{2^k-1}$ within the region $t-s \geq 0$; and 
\item The $(V^{C_{2^k}} + t-s, s)$-graded page of $\SliceSS(\BPCnminuskplusone \langle I \rangle)$ on or above the horizontal line $s = C_{V, k}$ within the region $t -s \geq 0$.  
\end{enumerate}
Here, the line $\mathcal{L}_{2^k-1}^V$ is a line of slope $(2^k-1)$ as defined in \cref{df:LineLVh-1}, and $C_{V, k}$ is a constant as defined in \cref{df:ConstantCinTranschromaticTheorem}. 
\end{thmx}

A more detailed description of the shearing isomorphism in \cref{thm:introThmTranschromaticIsomBPGI} is as follows: let $\E_{\Cnplusone}$ and $\E_{\Cnminuskplusone}$ represent the slice spectral sequences of  $\BPCnplusone \langle I \rangle$ and $\BPCnminuskplusone \langle I \rangle$, respectively.  Then the following statements hold: 
\begin{enumerate}
\item In the specified regions of their respective $\mathcal{E}_2$-pages, there is an isomorphism 
\begin{equation}\label{eq:IntroTranschromaticMainIsom}
\E_{\Cnplusone, 2}^{s', V+t'} \cong \Pb_{\Cnplusone/\Ck}\left(\E_{\Cnminuskplusone, 2}^{s, V^{\Ck}+t}\right),
\end{equation}
where
\begin{align*}
t' &= \left(|V^{\Ck}|\cdot 2^k-|V|\right) + 2^k t, \\
s' &= \left(|V^{\Ck}|\cdot 2^k-|V|\right) + (2^k-1)(t-s) +2^k s.
\end{align*}

\item In the specified region for $\E_{\Cnplusone}$, all the differentials have lengths of the form 
\[2^kr - (2^k-1), \,\,\, r \geq 2.\] 

\item The isomorphism (\ref{eq:IntroTranschromaticMainIsom}) induces a one-to-one correspondence between the $d_{2^kr - (2^k-1)}$-differentials in $\E_{\Cnplusone}$ originating on or above the line $\mathcal{L}^V_{2^k-1}$ and the $d_r$-differentials in $\E_{\Cnminuskplusone}$ originating on or above the horizontal line $s = C_{V, k}$. 
\end{enumerate}

In particular, when we set $V = 0$ in \cref{thm:introThmTranschromaticIsomBPGI}, a shearing isomorphism ${d_{2^kr - (2^k-1)} \leftrightsquigarrow d_r}$ is established between the following regions of spectral sequences:
\begin{enumerate}
\item The slice spectral sequence of $\BPCnplusone \langle I \rangle$ on or above the line $\mathcal{L}_{2^k-1}$ (the line of slope $(2^k-1)$ through the origin); and 
\item The entire slice spectral sequence of $\BPCnminuskplusone \langle I \rangle$.  
\end{enumerate}
By setting $I = \{1, 2, \ldots, m\}$, \cref{thm:introThmTranschromaticIsom} immediately follows from \cref{thm:introThmTranschromaticIsomBPGI}.  As an important application, it is an immediate consequence of \cref{thm:introThmTranschromaticIsom} that in the slice spectral sequence of $\BPCeight \langle 1 \rangle$ ($E_4^{hC_8}$), which detects all the Kervaire invariant elements, the differentials on or above the line of slope 1 are completely determined by the differentials in the slice spectral sequence of $\BPCfour \langle 1 \rangle$ ($E_2^{hC_4}$).  

Another notable case is when we set $I = \{m\}$.  In this case, the spectrum $\BPCn \langle I \rangle$ becomes
\[\BPCn \langle m, m \rangle := \BPCn\langle m \rangle/(\Cn \cdot \tee_1^{\Cn}, \ldots, \Cn \cdot \tee_{m-1}^{\Cn}).\]
These theories have been explored in \cite{BHLSZ}.  By setting $ I = \{m\}$, \cref{thm:introThmTranschromaticIsomBPGI} also establishes a transchromatic isomorphism theorem for the theories $\BPCnplusone \langle m, m \rangle$. 

We would like to emphasize the crucial role played by the pullback functor $\Pb_{\Cnplusone/\Ck}(-)$, which is the right-adjoint to the geometric fixed point functor $\Phi^{\Ck}(-)$.  This functor is crucial in establishing the equivalence in \cref{thm:introThm5DualTowerEquivalence} and the shearing isomorphism in \cref{thm:introThmTranschromaticIsomBPGI}.  While comparing the slice spectral sequence of a $\Cnplusone$-spectrum with a $\Cnminuskplusone$-spectrum, it may be tempting to consider $\Cnminuskplusone$ as a subgroup of $\Cnplusone$.  However, a key observation is that, to establish the desired equivalence, one should instead treat $\Cnminuskplusone$ as the quotient group $\Cnplusone/\Ck$. 

The shearing isomorphism of \cref{thm:introThmTranschromaticIsomBPGI} induces an explicit correspondence between classes located in the isomorphism regions.

\begin{thmx}[Correspondence formulas, \cref{thm:CorrespondenceFormula}] \label{thm:introThmCorrespondenceFormulas}
 We have the following correspondences between classes in $\SliceSS(\BPCnplusone \langle I \rangle)$ and $\SliceSS(\BPCnminuskplusone \langle I \rangle):$
\begin{enumerate}
\item Suppose $V \in RO(\Cnplusone/\Ck) = RO(\Cnminuskplusone)$ is an actual representation, then we have the correspondence 
\[a_V \leftrightsquigarrow a_V.\]
\item Suppose $V \in RO(\Cnplusone/\Ck) = RO(\Cnminuskplusone)$ is an orientable representation, then we have the correspondence
\[u_V \leftrightsquigarrow u_V. \]
\item For all $i \in I$ and $1 \leq j \leq n -k +1$, we have the correspondence 
\[N_{C_2}^{C_{2^{k+j}}}(\ti^{\Cnplusone}) \cdot \left(\frac{a_{\rhobar_{2^{k+j}}}}{a_{\rhobar_{2^{k+j}/2^k}}} \right)^{2^i-1} \leftrightsquigarrow N_{C_2}^{C_{2^j}}(\ti^{\Cnminuskplusone}).\]
\end{enumerate}
In (1) and (2), $V$ is also viewed as an element of $RO(\Cnplusone)$ through pullback along the map $\Cnplusone \to \Cnplusone/\Ck$.  In (3), the representation $\rhobar_{2^{k+j}/2^k} \in RO(C_{2^{k+j}})$ is defined as in \cref{df:RhobarNotation}.
\end{thmx}

The correspondence formulas in \cref{thm:introThmCorrespondenceFormulas} allows us to directly translate differentials and permanent cycles from the slice spectral sequence of $\BPCnminuskplusone \langle I \rangle$ to the slice spectral sequence of $\BPCnplusone \langle I \rangle$.  They play a key role in establishing the Periodicity Theorem (\cref{thm:introThmPeriodicity}) and the Vanishing Lines Theorem (\cref{thm:introThmVanishingLines}).

\subsection{Open questions and future directions}\label{subsec:OpenQuestions}

\subsubsection*{Other types of transchromatic phenomena}

In addition to the transchromatic phenomenon proved in \cref{thm:introThmTranschromaticIsom}, the following conjecture offers another transchromatic approach to studying the $\Cn$-fixed points of $E_h$ across different heights:

\begin{conj} \label{conj:BPCnheightDifferentials}
The differentials in the slice spectral sequence of $\BPCn$ can be classified by ``height''.  With this classification, the $\Cn$-slice spectral sequence of $\BPCn \langle m \rangle$ contains all the differentials in the slice spectral sequence of $\BPCn$ that are of height at most $h = 2^{n-1}m$.    
\end{conj}

To prove \cref{conj:BPCnheightDifferentials}, it is worth noting that the map
\begin{equation} \label{eq:BPCnmtoBPCnm-1}
\SliceSS(\BPCn \langle m+1 \rangle) \longrightarrow \SliceSS(\BPCn \langle m \rangle)
\end{equation}
is a surjection on the $\mathcal{E}_2$-page for all $m, n \geq 1$. 

\begin{conj} \label{conj:BPCnmtoBPCnm-1surjective}
For all $m \geq 1$ and $r \geq 2$, the map (\ref{eq:BPCnmtoBPCnm-1}) induces a surjection on the $\mathcal{E}_r$-page and of the $d_r$-differentials.
\end{conj}

Proving \cref{conj:BPCnmtoBPCnm-1surjective} would imply that the higher-height theories $E_{(m+1) \cdot 2^{n-1}}^{hC_{2^n}}$ inherit all the differentials from the lower-height theories $E_{m \cdot 2^{n-1}}^{hC_{2^n}}$.  This constitutes a significant step toward fully resolving \cref{conj:BPCnheightDifferentials}.  Currently, \cref{conj:BPCnmtoBPCnm-1surjective} has been verified for all the known computations listed in \cref{subsec:Motivation}.

Utilizing \cref{thm:introThmTranschromaticIsom}, there is a possible inductive approach to proving \cref{conj:BPCnmtoBPCnm-1surjective}.  The base case of the induction is when the group is $C_2$.  In this scenario, the surjectivity of (\ref{eq:BPCnmtoBPCnm-1}) has already been established for all $m \geq 1$.  Now, suppose that we have proved the claim for the group $\Cnminusone$.  For $\Cn$, we have the following diagram: 
\[\begin{tikzcd}
\SliceSS(\BPCn \langle m+1 \rangle) \ar[r, "\mathbf{1}"] \ar[d, "\text{shearing}", leftrightsquigarrow] & \SliceSS(\BPCn \langle m \rangle) \ar[d, "\text{shearing}", leftrightsquigarrow] \\
\SliceSS(\BPCnminusone \langle m+1 \rangle) \ar[r, "\mathbf{2}"] & \SliceSS(\BPCnminusone \langle m \rangle).
\end{tikzcd}\]
The vertical arrows in the diagram correspond to the shearing isomorphisms established in \cref{thm:introThmTranschromaticIsom}.  Given that we have already demonstrated that map $\mathbf{2}$ satisfies the claim outlined in \cref{conj:BPCnmtoBPCnm-1surjective}, it suffices to prove surjectivity of map $\mathbf{1}$ for the region below the line of slope 1.

\subsubsection*{Transchromatic isomorphism for the negative cone}

For any $G$-spectrum $X$, its equivariant slice spectral sequence is concentrated in two conical regions.  On the integer-graded page, one of these regions is in the first quadrant, and is known as the ``positive cone", while the other, located in the third quadrant, is referred to as the ``negative cone''.  Both regions are bounded by lines of slopes $0$ and $(|G|-1)$.

For connective theories such as $\BPG \langle m \rangle$, the slice spectral sequence will only have a positive cone.  However, for periodic theories such as $D^{-1}\BPG \langle m \rangle$ and Lubin--Tate theories, a negative cone also appears. 

On the integer-graded page, the behavior of the positive cone can be completely understood by utilizing the localized slice spectral sequence, as demonstrated in our use of the stratification tower and our proof of \cref{thm:introThmTranschromaticIsom}.  Given this result, it is reasonable to expect that the negative cone should also exhibit similar transchromatic phenomena.  Investigating the negative cone is particularly important in the context of duality.  Explicit computations have unveiled duality phenomena in the $C_2$-slice spectral sequences of $E_\mathbb{R}(n) := \bar{v}_n^{-1} \BPR \langle n \rangle$ \cite{GreenleesMeierDuality} and in the $C_4$-slice spectral sequence of $D^{-1}\BPCfour \langle 1 \rangle$ \cite{HHRKH}.

In collaboration with Yutao Liu and Guoqi Yan, we will analyze the isomorphism regions in the generalized Tate spectral sequence with respect to different families $\mathcal{F}$.  This analysis will provide a stratification for the negative cone of the slice spectral sequence for non-connective theories.  We will also leverage the shearing isomorphisms already established in this paper to deduce further shearing isomorphisms and vanishing line results for the negative cone of the slice spectral sequence.




\subsubsection*{\texorpdfstring{$Q_8$}{text}-equivariant transchromatic isomorphisms}

According to the classification of maximal finite 2-groups in the stabilizer group $\mathbb{S}_h$ \cite{Hewett, Hewett2, Bujard}, there are two possible types: cyclic groups of order a power of 2 and $Q_8$.  In this paper, our focus has been on cyclic groups.

\begin{quest}
Are there $Q_8$-equivariant analogues of the Transchromatic Isomorphism Theorem (\cref{thm:introThmTranschromaticIsom})? 
\end{quest}

This question is particularly relevant when the height $h$ is congruent to $2$ modulo $4$.  A crucial step in the proof of \cref{thm:introThmTranschromaticIsom} relies on analyzing the stratification tower of $\BPCnplusone \langle m \rangle$.  In the context of $Q_8$, the work of \cite{HahnShi} also implies the existence of a $Q_8$-equivariant orientation $\BPQeight \to E_h$. 

It would be interesting to attempt to create chromatically meaningful quotients from $\BPQeight$ and study their stratification towers.  Drawing a parallel to the comparison made with cyclic groups, where we compared the stratification towers of quotients of $\BPG$ and $\BPGmodtwo$, one might expect similar isomorphism results between the stratification towers of $\BPQeight$ and $\BPCtwoCtwo$, given that $Q_8/C_2 = C_2 \times C_2$. Interestingly, $C_2 \times C_2$ is not a finite subgroup of $\mathbb{S}_{h/2}$, making it intriguing to explore where the transchromatic phenomenon may manifest itself. 

\subsubsection*{Transchromatic phenomena at odd primes, other equivariant filtrations}
When $p$ is an odd prime and the height $h$ is of the form $m \cdot (p-1)p^{n-1}$, the Lubin--Tate spectrum $E_h$ admits a $C_{p^{n}}$-action, as shown in \cite{Hewett, Hewett2, Bujard}.  Consequently, at height $h/p = m \cdot (p-1)p^{n-2}$, the spectrum $E_{h/p}$ admits a $C_{p^{n-1}}$-action. Given this, the authors believe that there should be analogues of the Transchromatic Isomorphism Theorem (\cref{thm:introThmTranschromaticIsom}) for odd primes. In the following conjecture, we will leave certain terms vague in our statement, as discussed below.

\begin{conj} \label{conjecture:oddPrimaryShearing}
There is a shearing isomorphism $d_{p\cdot r-1} \leftrightsquigarrow d_r$ between the following regions of spectral sequences: 
\begin{enumerate}
\item The $C_{p^n}$-spectral sequence of $E_h$ on or above the line of slope $(p-1)$ within the region $t-s \geq 0$; and 
\item The $(C_{p^n}/C_p)$-spectral sequence of $E_{h/p}$ within the region $t-s \geq 0$.  
\end{enumerate} 
\end{conj}

Should \cref{conjecture:oddPrimaryShearing} hold true, an important calculation is to determine the $C_9$-fixed points of $E_6$ at the prime 3.  The theory $E_6^{hC_9}$ could play a crucial role in resolving the 3-primary Kervaire invariant problem.  Utilizing the shearing isomorphism outlined in \cref{conjecture:oddPrimaryShearing}, the computation of $E_2^{hC_3}$ by Hopkins and Miller (\cite{Nave}) can be inductively used to compute $E_6^{hC_9}$.  This approach may help simplify its computation.

An essential aspect of \cref{conjecture:oddPrimaryShearing} is the need to specify \textit{which} equivariant spectral sequence we are using to compute $E_h$.  In contrast to the situation at the prime 2, odd primes may offer alternative equivariant spectral sequences that are more efficient for computations.  It will be interesting to investigate how to construct a $C_{p^n}$-filtration for $E_h$ whose associated spectral sequence satisfies the shearing isomorphism proposed in \cref{conjecture:oddPrimaryShearing}.

In a more general context, for any equivariant filtration $\{P^\bullet\}$, we can create a new filtration $\{\EF \wedge P^\bullet\}$ by smashing it with $\EF$.  The resulting map of towers $\{P^\bullet\} \to \{\EF \wedge P^\bullet\}$ induces a map of the corresponding spectral sequences.  However, it is important to note that we are not always guaranteed that this map will produce isomorphism regions.  One crucial fact that establishes the existence of the stratification tower and the Transchromatic Isomorphism Theorem is the observation that this map indeed yields isomorphism regions when ${P^\bullet}$ is the equivariant slice filtration, as demonstrated by the Slice Recovery Theorem (\cite[Theorem~3.3]{MeierShiZengStratification}).

\begin{quest}
For what other equivariant filtrations are there stratification results and shearing isomorphisms? 
\end{quest}

\subsection{Organization of the paper}

In \cref{sec:StratificationReview}, we revisit the construction of the stratification tower and its relevant properties from \cite{MeierShiZengStratification}.  In \cref{sec:ShearingDifferentials}, we discuss the pullback functor, the dual slice tower, and the shearing isomorphism.  We prove a result (\cref{thm:ShearingGeneralTheory}) on shearing isomorphism between two interrelated spectral sequences.  This relationship emerges through the interplay of a pullback operation and a modification of filtrations. 

In \cref{sec:IsomLocalizedDualTowers}, we compare the localized dual slice towers of $\BPCnplusone \langle I \rangle$ and $\BPCnminuskplusone \langle I \rangle$, proving \cref{thm:introThm5DualTowerEquivalence} (\cref{thm:DualTowerEquivalenceMain}).  Combined with \cref{thm:ShearingGeneralTheory}, the equivalence of localized dual slice towers induces a shearing isomorphism of the corresponding localized slice spectral sequences.  In \cref{sec:CorrespondenceFormulas}, we establish correspondence formulas (\cref{thm:CorrespondenceFormula}) for classes in these localized slice spectral sequences.

In \cref{sec:TranschromaticIsomTheorem}, we prove \cref{thm:introThmTranschromaticIsomBPGI} (\cref{thm:TranschromaticMain}) and state several corollaries that will be useful for our discussions in the subsequent sections.  \cref{thm:introThmCorrespondenceFormulas} is a direct consequence of \cref{thm:CorrespondenceFormula} and \cref{thm:TranschromaticMain}.  We also provide an alternative proof of the Slice Differentials Theorem by Hill, Hopkins, and Ravenel (\cref{thm:HHRSliceDiffTheorem}).  

Our main theorem, \cref{thm:introThmTranschromaticIsom}, directly follows from \cref{thm:introThmTranschromaticIsomBPGI} by setting $I = \{1, 2, \ldots, m\}$.  This is discussed in more detail in \cref{sec:TranschromaticTowerLubinTateTheories}.  

In the final two sections, we turn our attention to applications of \cref{thm:introThmTranschromaticIsom}.  In \cref{sec:ROGPeriodicity}, we prove \cref{thm:introThmPeriodicity} (\cref{thm:ROGPeriodicityTranschromatic}), which establishes $RO(G)$-graded periodicities of height-$h$ theories from $RO(G/{C_2})$-graded periodicities of height-$(h/2)$-theories.  In \cref{sec:VanishingLines}, we prove \cref{thm:introThmVanishingLines} (\cref{thm:VanishingLineGeneralSlope}).

\subsection*{Acknowledgements}
We would like to thank William Balderrama, Agn\`{e}s Beaudry, Mark Behrens, Mike Hopkins, Hana Jia Kong, Tyler Lawson, Guchuan Li, Wenao Li, Yutao Liu, Peter May, Doug Ravenel, Guozhen Wang, Guoqi Yan, and Foling Zou for helpful conversations.  We would like to especially thank Mike Hill and Zhouli Xu for comments on an earlier draft of our paper.  The first author is supported by the NWO grant VI.Vidi.193.111 and the second author is supported in part by NSF Grant DMS-2313842.

\subsection*{Notations and conventions}
\begin{enumerate}
\item The slice tower we consider will always be the \emph{regular} slice towers (see \cite{UllmanPaper, UllmanThesis}).  
\item For a $G$-spectrum $X$, we denote by $\pi_*^G X$ the integer-graded homotopy group of $X$; by $\pi_\star^GX$ the $RO(G)$-graded homotopy group of $X$; by $\Mpi_* X$ the integer-graded Mackey functor valued homotopy group of $X$; and by $\Mpi_\star X$ the $RO(G)$-graded Mackey functor valued homotopy groups of $X$.  
\item All of our spectral sequences are $RO(G)$-graded.  For a fixed $V \in RO(G)$, the ``$(V+t-s,s)$-graded page'' is the part of the spectral sequence consisting of all the classes with degrees $(V+t-s, s)$, where $t, s \in \mathbb{Z}$.  The ``integer-graded page'' is when we set $V = 0$. 
\item We will denote the regular representation of $\Cn$ by $\rho_{2^n}$.  
\item For $0 \leq i \leq n-1$, we will denote the real 2-dimensional $\Cn$-representation corresponding to rotation by $\left(\frac{\pi}{2^i}\right)$ by $\lambda_i$.  With this notation, $\lambda_0 = 2\sigma$, where $\sigma$ is the 1-dimensional sign representation. 
\item We will denote by $\Sp_G$ the homotopy category of genuine $G$-spectra and by $\wedge$ the derived smash product on it.
\end{enumerate} 

\section{Stratification of the slice spectral sequence}\label{sec:StratificationReview}

In this section, we will recall the relevant definitions and results in \cite{MeierShiZengStratification} that will be needed in the following sections.  The construction of the stratification tower in \cite{MeierShiZengStratification} relies on the utilization of the localized slice spectral sequence, which is a modification of the slice spectral sequence studied by the authors in \cite{MeierShiZengHF2}.  

Let $G$ be a finite group, $X$ a $G$-spectrum, and $\FF$ a family of subgroups of $G$.  The localized slice spectral sequence of $X$ with respect to $\mathcal{F}$ is the spectral sequence associated with the tower $\EF \wedge P^\bullet X$, obtained by smashing the universal space $\EF$ with the slice tower of $X$.  This spectral sequence is denoted by $\EF \wedge \SliceSS(X)$.  

Whenever there is an inclusion $\mathcal{F} \subset \mathcal{F}'$ of families, it induces a map $\EF \to \EF'$ of universal spaces, consequently leading to a map 
\[\varphi: \EF \wedge \SliceSS(X) \longrightarrow \EF' \wedge \SliceSS(X)\]
of the corresponding localized slice spectral sequences.

\begin{construction}[Stratification tower]\rm \label{construction:Tower}
Let $\mathcal{F}_{\leq h}$ be the family consisting of all subgroups of $G$ of order $\leq h$.  Note from definition, $\mathcal{F}_{\leq 0}$ is the empty family, and $\mathcal{F}_{\leq |G|}$ is the family consisting of all subgroups of $G$.  The chain of inclusions 
\[\FF_{\leq 0} \subset \FF_{\leq 1} \subset \cdots \subset \FF_{\leq |G|}\]
induces the maps
\[\EF_{\leq 0} \longrightarrow \EF_{\leq 1} \longrightarrow \cdots \EF_{\leq |G|-1} \longrightarrow \EF_{\leq |G|}.\]
Since $\EF_{\leq 0} \simeq S^0$ and $\EF_{\leq |G|} \simeq *$, the maps above induces a decreasing filtration of the equivariant slice spectral sequence of $X$ given by the tower 
\begin{equation}\label{eq:StratificationTower}
\left\{\EF_{\leq \bullet} \wedge \SliceSS(X), \, 0 \leq \bullet \leq |G| \right\}    
\end{equation} 
of localized slice spectral sequences.  This is called the \textit{\color{darkcyan} stratification tower}. 
\end{construction}

\begin{df} \rm \label{df:HmaxHmintau}
Suppose $V \in RO(G)$, and $S$ is a collection of subgroups of $G$.  Define 
\[\tau_V(S) = \max_{H \in S} \left(|V^H| \cdot |H| - |V|\right).\]
\end{df} 

\begin{df} \rm \label{df:LineLVh-1}
For a fixed $V \in RO(G)$ and $h \geq 1$, let $\mathcal{L}_{h-1}^V$ represent the line of slope $(h-1)$ on the $(V+t-s, s)$-graded page that is defined by the equation 
\[s = (h-1)(t-s) + \tau_V(\mathcal{F}_{\leq h}).\]
\end{df}

In \cref{df:LineLVh-1}, there are two special cases worth noting: 
\begin{enumerate}
\item When $V = 0$, we have $\tau_0(\mathcal{F}_{\leq h})=0$, and $\mathcal{L}_{h-1}^0$ is the line of slope $(h-1)$ through the origin on the integer-graded page.  In this case, we will denote this line by $\mathcal{L}_{h-1}$.  
\item When $h = 1$, we have $\tau_V(\mathcal{F}_{\leq 1}) = |V^e| \cdot 1 - V = 0$.  In this case, the line $L^V_0$ is the horizontal line $s = 0$.
\end{enumerate}

\begin{thm}[Slice Recovery Theorem, \cite{MeierShiZengStratification} Theorem 3.3]\label{thm:SliceRecovery1}
Let $X$ be a $G$-spectrum, and let $h \geq 1$.  On the $(V+t-s, s)$-graded page where $t-s \geq 0$, the map
\[\varphi_h: \SliceSS(X)  \longrightarrow \EF_{\leq h} \sm \SliceSS(X)\]
induces an isomorphism of spectral sequences on or above the line $\mathcal{L}_{h-1}^V$.  In the region where $t-s < 0$, the map $\varphi_h$ induces an isomorphism of spectral sequences on or above the horizontal line $s= \tau_V(\mathcal{F}_{\leq h})$.
\end{thm}

\begin{cor}[\cite{MeierShiZengStratification} Corollary 3.4]\label{cor:sliceRecoveryExample1}
The map 
\[\varphi_1: \SliceSS(X) \longrightarrow \widetilde{E}G \wedge \SliceSS(X)\]
 induces an isomorphism of spectral sequences on or above the horizontal line $s = 0$ on all the $(V+t-s, s)$-graded pages.  
\end{cor}

When $G = C_{2^{n+1}}$, the subgroups of $G$ are linearly ordered by inclusion.  They are of the form $C_{2^k}$ for $0 \leq k \leq n+1$.  For all $k$ and $\ell$ such that $1 \leq k \leq n+1$ and $2^{k-1}\leq \ell \leq 2^k-1$, we have the equality $\mathcal{F}_{\leq \ell} = \mathcal{F}[C_{2^k}]$, where $\mathcal{F}[\Ck]$ is the family consisting of all subgroups of $\Cnplusone$ that do not contain $\Ck$.  The chain of inclusions 
\[\FF_{\leq 0} \subset \FF_{\leq 1} \subset \cdots \FF_{\leq |G|-1}\]
is essentially 
\[\varnothing \subset \mathcal{F}[C_2] \subset \mathcal{F}[C_{4}] \subset \cdots \subset \mathcal{F}[C_{2^{n+1}}].\]
For any $\Cnplusone$-spectrum $X$ and a subgroup $\Ck \subseteq \Cnplusone$, there is a residual $(C_{2^{n+1}}/C_{2^k})$-action on the $C_{2^k}$-geometric fixed points $\Phi^{C_{2^k}}(X)$, and we also have the equivalence 
\[(\EF[C_{2^k}] \wedge X)^{C_{2^{n+1}}} \simeq  \left(\Phi^{C_{2^k}}(X)\right)^{C_{2^{n+1}}/C_{2^k}}.\]
The stratification tower (\ref{eq:StratificationTower}) becomes the following: 

\begin{equation} \label{diagram:SliceSSTower}
\begin{tikzcd}
\SliceSS(X) \ar[rr, "\mathcal{L}^V_0"] \ar[rrd, "\mathcal{L}^V_{1}"] \ar[rrddd, bend right = 10, "\mathcal{L}^V_{2^{n-2}-1}"] \ar[rrdddd, bend right = 30, "\mathcal{L}^V_{2^{n-1}-1}",swap]&& \EF[C_2] \wedge \SliceSS(X) \ar[d, "\mathcal{L}^V_{1}"] \ar[r, Rightarrow] & \Mpi_\star \Phi^{C_2}(X) \ar[d] \\
&& \EF[C_4] \wedge \SliceSS(X) \ar[d, "\mathcal{L}^V_{3}"] \ar[r,Rightarrow] & \Mpi_\star \Phi^{C_{4}}(X) \ar[d]\\
&& \vdots \ar[d, "\mathcal{L}^V_{2^{n-1}-1}"] & \vdots \ar[d] \\ 
&& \EF[C_{2^{n}}] \wedge \SliceSS(X) \ar[d, "\mathcal{L}^V_{2^{n}-1}"] \ar[r, Rightarrow] & \Mpi_\star \Phi^{C_{2^{n}}}(X) \ar[d] \\
&& \EF[\Cnplusone] \wedge \SliceSS(X) \ar[r,Rightarrow] & \pi_*\, \Phi^{C_{2^{n+1}}}(X).
\end{tikzcd}
\end{equation}

\begin{figure}
\begin{center}
\makebox[\textwidth]{\hspace{-0.5in}\includegraphics[trim={0cm 0cm 0cm 0cm}, clip, scale = 0.6]{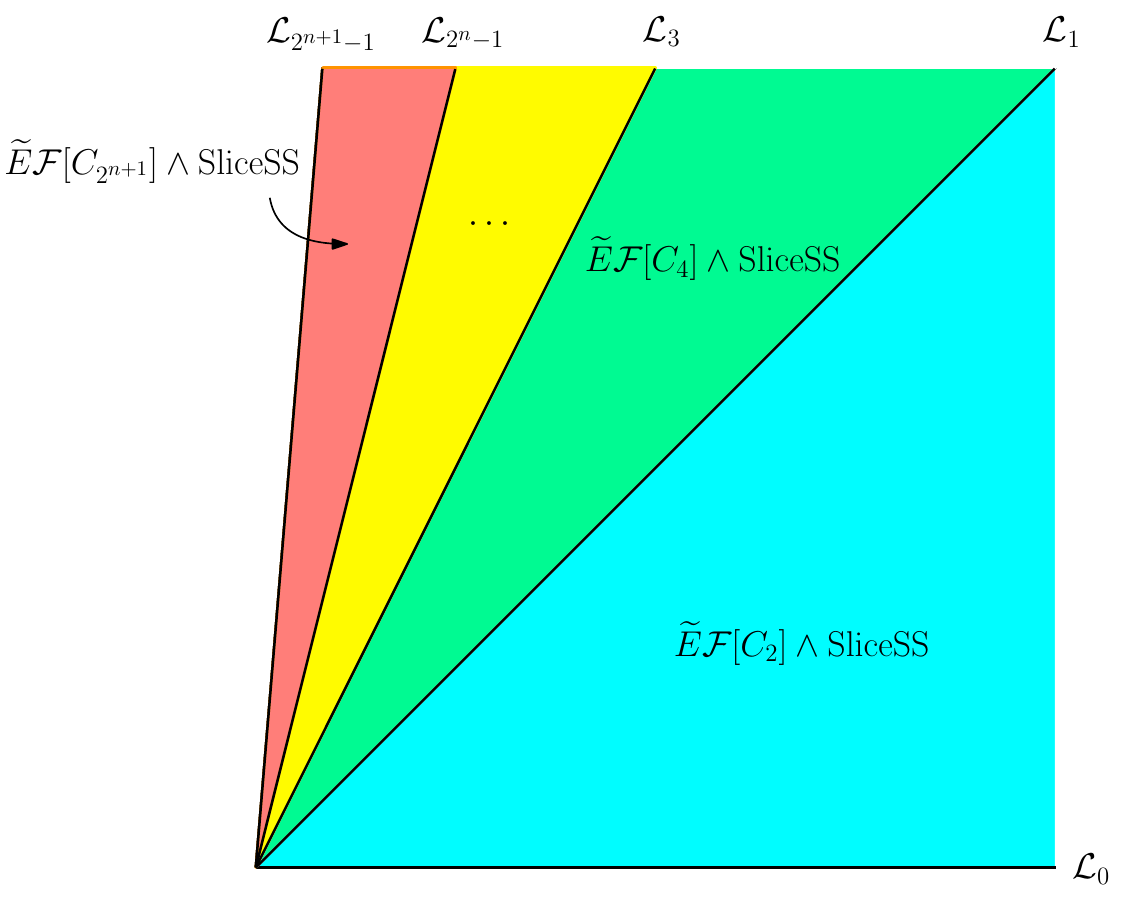}}
\caption{The stratification tower of the slice spectral sequence.}
\hfill
\label{fig:StratificationTower}
\end{center}
\end{figure}

\cref{fig:StratificationTower} shows a picture of the stratification tower when $V = 0$.  In the stratification tower above, the maps are labeled by the lines $\mathcal{L}^V_{h-1}$, indicating the isomorphism regions.  The slice spectral sequence of $X$ is stratified into different regions separated by the lines $\mathcal{L}^V_{1}$, $\mathcal{L}^V_{3}$, $\ldots$, and $\mathcal{L}^V_{2^{n}-1}$.  The differentials within these regions can be recovered from the localized slice spectral sequences, which compute the geometric fixed points equipped with the residue $(C_{2^{n+1}}/C_{2^k})$-action.  As we move up the tower, \cref{thm:SliceRecovery1} shows that each localized slice spectral sequence contains increasingly more information about the original slice spectral sequence.

\section{Shearing isomorphisms of spectral sequences}
\label{sec:ShearingDifferentials}
In this section, we will discuss the pullback functor, the dual slice tower, and the shearing isomorphism.  We will prove a result (\cref{thm:ShearingGeneralTheory}) on shearing isomorphism between two interrelated spectral sequences.  This relationship emerges through the interplay of a pullback operation and a modification of filtrations. 

\subsection{The pullback functor}
Suppose $G$ is a finite group and $N 
\subseteq G$ is a normal subgroup.  The $N$-geometric fixed points functor 
\[\Phi^N: \Sp_G \longrightarrow \Sp_{G/H}\]
is the functor that sends a $G$-spectrum $X$ to $\Phi^N(X) = (\EF[N] \wedge X)^N$, its $N$-geometric fixed points.  This functor is a left adjoint functor.  

\begin{df}\rm \label{df:PullBack}
The \textit{pullback functor}
\[\Pb_{G/N}: \Sp_{G/N} \longrightarrow \Sp_G \]
is the right adjoint to the $N$-geometric fixed points functor $\Phi^N$. 
\end{df}

The following explicit construction of the pullback functor $\Pb_{G/N}$ can be found in \cite[II. $\S$9]{LewisMaySteinberger} and \cite[Section~4]{HillPrimer}: let $q^*: \Sp_{G/N} \to \Sp_{G}$ denote the inflation functor that is associated to the quotient map $q: G \to G/N$.  Then the pullback functor sends a $G/N$-spectrum $X$ to $\Pb_{G/N}(X) = \EF[N] \wedge q^*X$.  

Several pleasant properties of the pullback functor $\Pb_{G/H}$ are shown in \cite[Section~4.1]{HillPrimer} (where $\Pb_{G/N}$ is denoted by $\phi_N^*$).  In particular, the $N$-fixed points functor $(-)^N$ establishes an equivalence between $G/N$-spectra and the image of $\Pb_{G/N}$, which are $G$-spectra of the form $\EF[N] \wedge X$ (\cite[Proposition~4.3]{HillPrimer}). 

There are two lemmas that will be useful for our subsequent discussions in later sections.  

\begin{lem}\label{lem:easyFact1}
For a $G$-spectrum $X$ and a normal subgroup $N$ of $G$, the following equivalence holds:
\[\EF[N] \wedge X \simeq \Pb_{G/N}\Phi^N(X).\]
\end{lem}
\begin{proof}
This follows immediately from the discussion above.  Since $\EF[N] \wedge X$ is in the image of $\Pb_{G/N}$, 
\begin{align*}
\EF[N] \wedge X &\simeq \Pb_{G/N}(\EF[N] \wedge X)^N \\
&\simeq \Pb_{G/N}\Phi^N(X). \qedhere
\end{align*}
\end{proof}

\begin{df}[\cite{GreenleesMay, HillPrimer}]\rm
Suppose $\underline{M}$ is a Mackey functor for $G/N$, then $\Pb_{G/N}\underline{M}$ is the Mackey functor on $G$ obtained by composing $\underline{M}$ with the $N$-fixed points functor (that sends finite $G$-sets to finite $G/N$-sets).  More precisely, 
\[\Pb_{G/N}\underline{M}(G/H) = \left\{\begin{array}{ll} 0 & \text{if } N \not \subseteq H \,\,\, ( H \in \mathcal{F}[N]),\\
\underline{M}((G/N)/(H/N)) &\text{if } N \subseteq H \,\,\,( H \notin \mathcal{F}[N]). \end{array} \right.\]
\end{df}

\begin{lem}\label{lem:PullbackHomotopyGroups}
Suppose $X$ is a $G/N$-spectrum.  Then for every $V \in RO(G)$, there is a canonical isomorphism 
\[\Mpi_V \Pb_{G/N}(X) \cong \Pb_{G/N} (\Mpi_{V^N} X),\]
where $V^N$ is the $N$-fixed points of $V$, considered as an element in $RO(G/N)$.  
\end{lem}
\begin{proof}
Let $H$ be a subgroup of $G$.  If $N \not\subseteq H$, then $i_H^* \EF[N]$ is contractible, and so $i_H^* \Pb_{G/N}(X)$ is also contractible.  In this case, we have $\pi_{V}^H \Pb_{G/N}(X) \cong 0$.

If $N \subseteq H$, then $i_H^* \Pb_{G/N}(X)= \Pb_{H/N}(i_{H/N}^* X)$, and we have
\begin{align*}
\pi_{V}^H \Pb_{G/N}(X) &\cong  \left[i_H^* S^V, \Pb_{H/N}(i_{H/N}^* X) \right]^H \\
&\cong   \left[\Phi^N(S^{i_H^*V}), i_{H/N}^* X\right]^{H/N} \\
&\cong   \left[S^{(i_H^*V)^N}, i_{H/N}^* X\right]^{H/N} \\
&\cong   \left[S^{i_{H/N}^*V^N}, i_{H/N}^* X\right]^{H/N} \\
&\cong   \pi_{V^N}^{H/N} X.    \qedhere
\end{align*}
\end{proof}

\subsection{The dual slice tower}

For a $G$-spectrum $X$, let $P_{d+1}X$ be the fiber of $X \to P^d X$.  There is a functorial fiber sequence 
\[P_{d+1} X \longrightarrow X \longrightarrow P^d X.\] 
The tower
\[P_\bullet X = \left\{\cdots \longrightarrow P_{d+1}X \longrightarrow P_d X \longrightarrow P_{d-1} X \longrightarrow \cdots\right\}\]
is the \textit{dual slice tower} of $X$.  It is a consequence of \cite[Theorem~4.42]{HHR} that $\varinjlim P_\bullet X \simeq X$ and $\varprojlim P_\bullet X \simeq *$.  Just like the slice tower is an equivariant refinement of the Postnikov tower of $X$, the dual slice tower is an equivariant refinement of the Whitehead tower of $X$.

From the definition, the cofiber of $P_{d+1} X \to P_d X$ and the fiber of $P^d X \to P^{d-1} X$ are both $P_d^dX$, the $d$-slice of $X$.  The spectral sequence associated to $P_\bullet X$ has $\mathcal{E}_2$-page 
\[\E_2^{s, V} = \Mpi_{V-s} P^{|V|}_{|V|}X \Longrightarrow \Mpi_{V-s} X,\]
and is the exact same spectral sequence as the slice spectral sequence of $X$ (which is associated to $P^\bullet X$).  All the results and constructions that we have established for the slice tower carry over directly to the dual slice tower. 

Moving forward, we will be working with the dual slice tower.  This is just a stylistic preference, as all of our theorem statements and arguments will work analogously for the slice tower as well.

\begin{df}\rm \label{df:doubleTower}
The \textit{double} of a tower $P_\bullet$ is the tower $\D P_\bullet$ defined by setting $\D P_{2d-\varepsilon}= P_d$ for all $d$ and $\varepsilon = 0, 1$.  For $k \geq 1$, the \textit{$k$-fold double} of $P_\bullet$ is the tower $\D^k P_\bullet$, obtained by taking the double $k$-times.  In particular, we have $(\D^k P)^{2^kd}_{2^kd} = P^d_d$, and $(\D^k P)^{2^kd -j}_{2^kd-j}$ are contractible for all $k \geq 1$ and $1 \leq j \leq 2^k-1$.  
\end{df}

\begin{prop}[\cite{UllmanPaper} Corollary~4.5, \cite{HillPrimer} Remark~4.13] \label{prop:pullBackSliceTower}
Let $G$ be a finite group and $N \subseteq G$ a normal subgroup. 
 For $X$ a $G/N$-spectrum, the following equivalence holds: 
\[P_m (\Pb_{G/N} X) \simeq \Pb_{G/N}\left(P_{\lceil m/|N|\rceil} X \right).\]
In the special case that $N = C_{2^k}$, there is an equivalence
\[P_\bullet \left(\Pb_{G/C_{2^k}} X\right) \simeq \Pb_{G/C_{2^k}} \left(\D^k P_\bullet X\right).\]
\end{prop}

\subsection{Shearing isomorphism}
Let $G$ be a finite group, and let $N = C_{2^k}$ be a normal subgroup of $G$.  Suppose $P_\bullet$ is a $G/N$-equivariant tower.  Define the $G$-equivariant tower $Q_\bullet$ as follows: 
\[Q_\bullet = \Pb_{G/N}(\D^k P_\bullet).\]
The following theorem establishes the relationship between the $G$-equivariant spectral sequence $\E_{Q_\bullet}$ that is associated to $Q_\bullet$ and the $G/N$-equivariant spectral sequence $\E_{P_\bullet}$ that is associated to $P_\bullet$.  

\begin{thm} \label{thm:ShearingGeneralTheory}
Suppose $V \in RO(G)$, and let $V^N \in RO(G/N)$ be the $N$-fixed points of $V$.  There is a shearing isomorphism of spectral sequences
\[\E_{Q_\bullet, 2^kr - (2^k-1)}^{s', V+t'} \cong \Pb_{G/N}\left(\E_{P_\bullet, r}^{s, V^N+t}\right) \,\,\,\,\, (r \geq 2),\]
where
\begin{align*}
t' &= \left(|V^N|\cdot 2^k-|V|\right) + 2^k t, \\
s' &= \left(|V^N|\cdot 2^k-|V|\right) + (2^k-1)(t-s) +2^k s.
\end{align*}
In particular, this shearing isomorphism induces a one-to-one correspondence between the the $d_{2^kr - (2^k-1)}$-differentials in $\E_{Q_\bullet}$ and the $d_r$-differentials in $\E_{P_\bullet}$.  
\end{thm}
\begin{proof}
Set $W = V+t$. As elaborated in \cite[Section 3.3]{MeierShiZengHF2}, the $\E_r$-page of the spectral sequence associated with $P_{\bullet}$ is given by
\[\E_{P_{\bullet}, r}^{s,W^N} = \im \left(\Mpi_{W^N-s}P^{|W^N|+(r-2)}_{|W^N|} \longrightarrow  \Mpi_{W^N-s}P^{|W^N|}_{|W^N|-(r-2)} \right).\]
By the construction of the tower $Q_\bullet$ and \cref{lem:PullbackHomotopyGroups}, the Mackey functor $P^*_{G/N}\E_{P_{\bullet}, r}^{s, W^N}$ is isomorphic to 
\begin{equation}\label{eq:Q} \im \left(\Mpi_{W-s}Q^{|W^N|\cdot 2^k+2^k(r-2)}_{|W^N|\cdot 2^k} \longrightarrow  \Mpi_{W-s}Q^{|W^N|\cdot 2^k}_{|W^N|\cdot 2^k-2^k(r-2)} \right).\end{equation}
Note that we could replace $2^k(r-2)$ here by any larger number smaller than $2^k(r-1)$ since the terms in $Q_{\bullet}$ repeat $2^k$ times. The term (\ref{eq:Q}) is isomorphic to $\E_{Q_\bullet, r'}^{s', V+t'}= \E_{Q_\bullet, r'}^{s', W-t+t'}$ if $2^k(r-2) \leq r' -2 \leq 2^k(r-1)-1$ and the pair $(t', s')$ satisfies the following equalities: 
\begin{align*}
|W| -t + t'&= |W^N|\cdot 2^k, \\ 
W -t + t' -s' &= W -s.
\end{align*}
Equivalently, if $2^kr-(2^{k+1}-2) \leq r' \leq 2^kr-(2^k-1)$ and 
\begin{align*}
t' &= |W^N|\cdot 2^k - |W|+t = \left(|V^N|\cdot 2^k-|V|\right) + 2^k t \\ 
s' &= t' -t +s = \left(|V^N|\cdot 2^k-|V|\right) + (2^k-1)(t-s) +2^k s.
\end{align*}
Therefore, $\E_{Q_\bullet, r'}^{s', V+t'} \cong \Pb_{G/N}\left(\E_{P_\bullet, r}^{s, V^N+t}\right)$ for $r\geq 2$ with $r'$ falling within the specified range. 
    
The previous paragraph shows that unless $r'$ is of the form $2^kr-(2^k-1)$, $\E_{Q_\bullet, r'}^{s', V+t'} = \E_{Q_\bullet, r'+1}^{s', V+t'}$ and so $d_{r'}=0$ in $\E_{Q_\bullet}$.  We claim that when ${r' = 2^kr-(2^k-1)}$, the differential $d_{r'}$ in $\E_{Q_\bullet}$ corresponds to the differential $d_r\colon \E_{P_{\bullet}, r}^{s,W^N}\to \E_{P_{\bullet}, r}^{s+r, W^N+(r-1)}$. This $d_r$-differential is induced by the boundary map 
\[ \Mpi_{W^N-s}P^{|W^N|}_{|W^N|-(r-2)} \longrightarrow \Mpi_{W^N -s-1}P^{|W^N|+(r-1)}_{|W^N|+1}.\]
Upon applying $P^*_{G/N}$, this map becomes isomorphic to 
\[ \Mpi_{W-s}Q^{|W^N|\cdot 2^k}_{|W^N|\cdot 2^k-2^k(r-2)} \longrightarrow \Mpi_{W -s-1}Q^{|W^N|\cdot 2^k+2^k(r-1)}_{|W^N|\cdot 2^k+1}.\]
Using the formula for $t'$ and $s'$ above, this indeed yields the differential 
\[d_{r'}\colon \E_{Q_\bullet, r'}^{s', W+t'-t} \longrightarrow \E_{Q_\bullet, r'}^{s'+r', W+t'-t+(r'-1)}.\qedhere\]
\end{proof}

\section{Isomorphism of localized slice towers}
\label{sec:IsomLocalizedDualTowers}

In this section, we will compare the localized dual slice towers of $\BPCnplusone \langle I \rangle$ and $\BPCn \langle I \rangle$.  Our main result in this section the following theorem, which establishes an isomorphism between the various $C_{2^{n+1}}$-equivariant and $C_{2^n}$-equivariant localized dual slice towers via the pullback functor $\Pb_{C_{2^{n+1}}/C_2}(-)$ and the doubling operation $\D(-)$.

\begin{thm}\label{thm:DualTowerEquivalenceMain}
Suppose $n$ is a positive integer, and $I \subseteq \N$.  For all $1 \leq k \leq n$, we have the following equivalence of towers:
\[\EF[C_{2^{k+1}}] \wedge P_\bullet\left(\BPCnplusone \langle I \rangle \right) \simeq \Pb_{C_{2^{n+1}}/C_2}\D \left( \EF[C_{2^k}] \wedge P_\bullet (\BPCn \langle I \rangle) \right).\]
Moreover, by applying this equivalence iteratively $k$-times, we have the following equivalence of towers: 
\[\EF[C_{2^{k+1}}] \wedge P_\bullet\left(\BPCnplusone \langle I \rangle \right) \simeq \Pb_{C_{2^{n+1}}/C_{2^k}}\D^k\left(\EF[C_{2}] \wedge P_\bullet (\BPCnminuskplusone \langle I \rangle) \right).\]
\end{thm}

To prove \cref{thm:DualTowerEquivalenceMain}, we will begin by proving a lemma that establishes a connection between the geometric fixed points of the $\ti^{C_{2^n}}$-generators for $\pi_{*\rho_2}^{C_2} \BPCn$ and the $\ti^{C_{2^{n+1}}}$-generators for $\pi_{*\rho_2}^{C_2} \BPCnplusone$. 
 First, consider the $C_2$-equivariant map 
\[\ti^{C_{2^n}}: S^{(2^i-1)\rho_{2}} \longrightarrow i_{C_2}^*\BPCn.\]
After taking $C_2$-geometric fixed points $\Phi^{C_2}(-)$, we obtain the map 
\[\Phi^{C_2}(\ti^{C_{2^n}}): S^{2^i-1} \longrightarrow \Phi^{C_2}i_{C_2}^*\BPCn \simeq i_e^* \Phi^{C_2}N_{C_2}^{C_n}(\BPR) \simeq i_e^* N_e^{C_{2^{n-1}}} \HF.\]
Now, consider the $C_2$-equivariant map 
\[\ti^{C_{2^{n+1}}}: S^{(2^i-1)\rho_{2}} \longrightarrow i_{C_2}^*\BPCnplusone.\]
After taking the norm $N_{C_2}^{C_4}(-)$ and applying the norm-restriction adjunction, we obtain a $C_4$-equivariant map 
\[N_{C_2}^{C_4}(\ti^{C_{2^{n+1}}}): S^{(2^i-1)\rho_{4}} \longrightarrow N_{C_2}^{C_4}i_{C_2}^*\BPCnplusone \longrightarrow i_{C_4}^* \BPCnplusone.\]
Here, we would like to note that since $\MUCnplusone$ splits as a wedge of suspensions of $\BPCnplusone$, the norm-restriction adjunction for $\BPCnplusone$ is defined to be the composition map 
\[N_{C_2}^{C_4} i_{C_2}^*\BPCnplusone \longrightarrow N_{C_2}^{C_4} i_{C_2}^*\MUCnplusone \longrightarrow i_{C_4}^*\MUCnplusone \longrightarrow i_{C_4}^* \BPCnplusone.\]
Taking the $C_4$-geometric fixed points $\Phi^{C_4}(-)$ of the map $N_{C_2}^{C_4}(\ti^{C_{2^{n+1}}})$ defined above produces the map
\begin{equation*}
\begin{split}
\Phi^{C_4}N_{C_2}^{C_4}(\ti^{C_{2^{n+1}}}): S^{2^i-1} \longrightarrow \Phi^{C_4}i_{C_4}^* \BPCnplusone &\simeq i_e^* \Phi^{C_4} N_{C_4}^{C_{2^{n+1}}}(\BPCfour) \\
&\simeq i_e^* N_e^{C_{2^{n-1}}}\HF.
\end{split}
\end{equation*}

\begin{lem}\label{lem:tiGeneratorRelations}
For all $n \geq 1$ and $i\geq 1$, we have the equality
\[\Phi^{C_2}(\ti^{C_{2^n}}) = \Phi^{C_4} N_{C_2}^{C_4} (\ti^{C_{2^{n+1}}}).\]
\end{lem}
\begin{proof}
When $n = 1$, it is a consequence of \cite[Proposition~5.50]{HHR} that the maps $\Phi^{C_2}(\ti^{C_{2^n}})$ and $\Phi^{C_4}N_{C_2}^{C_4}(\ti^{C_4})$ are both 0 for all $i \geq 1$, and so the equality holds.  For $n \geq 2$, first consider the Real orientations 
\[\bar{x}_1, \bar{x}_2: \CP \longrightarrow \Sigma^{\rho_{2}} i_{C_2}^*\BPCn\]
that define the Real formal group laws $\bar{F}_1$ and $\bar{F}_2$ corresponding to the maps 
\[\bar{\iota}_1, \bar{\iota}_2: \BPR \longrightarrow i_{C_2}^* \BPCn \simeq \underbrace{\BPR \wedge \cdots \wedge \BPR}_{2^{n-1}},\]
where $\bar{\iota}_1$ and $\bar{\iota}_2$ are inclusions into the first and the second factors of $i_{C_2}^* \BPCn$, respectively.  By definition, the $\ti^{C_{2^n}}$-generators are the coefficients of the power series relating $\bar{x}_1$ and $\bar{x}_2$: 
\begin{equation}
\bar{x}_2  = \bar{x}_1 + {\sum_{i \geq 1}}^{\bar{F}_2} \ti^{C_{2^n}}\bar{x}_1^{2^i}. \label{eq:tiC2nDef}
\end{equation}
Note that $\Phi^{C_2}(\bar{\iota}_1) = \iota_1$ and $\Phi^{C_2}(\bar{\iota}_2) = \iota_2$, where 
\[\iota_1, \iota_2: \HF \longrightarrow i_{e}^* N_{e}^{C_{2^{n-1}}} \HF \simeq \underbrace{\HF \wedge \cdots \wedge \HF}_{2^{n-1}}\]
are again the inclusions into the first and the second factors, respectively. 
 It follows that $\Phi^{C_2}(\bar{x}_1) = a_1$ and $\Phi^{C_2}(\bar{x}_2) = a_2$, where $a_1$ and $a_2$ are the orientations 
\[a_1, a_2: \RP \to \Sigma^1 i_{e}^* N_{e}^{C_{2^{n-1}}} \HF\]
corresponding to the formal group laws with $MO$-orientations $\iota_1$ and $\iota_2$.  Since applying $\Phi^{C_2}(-)$ is a ring homomorphism \cite[Proposition~2.59]{HHR}, applying $\Phi^{C_2}(-)$ to (\ref{eq:tiC2nDef}) shows that the power series relating $a_1$ and $a_2$ is  
\begin{equation}
a_2 = f(a_1) = a_1 + {\sum_{i \geq 1}}^{F_2} \Phi^{C_2}(\ti^{C_{2^n}}) a_1^{2^i}. \label{eq:a1a2RelationOne}
\end{equation}
Here, $F_2$ is the formal group law corresponding to $\iota_2$. 

Now, consider the Real orientations 
\[\widetilde{x}_1, \widetilde{x}_2: \CP \longrightarrow \Sigma^{\rho_2} i_{C_2}^*\BPCnplusone\]
defining the Real formal group laws $\widetilde{F}_1$ and $\widetilde{F}_2$ corresponding to the maps 
\[\widetilde{\iota}_1, \widetilde{\iota}_2: \BPR \longrightarrow i_{C_2}^* \BPCnplusone \simeq \underbrace{\BPR \wedge \cdots \wedge \BPR}_{2^n}, \]
where $\widetilde{\iota}_1$ and $\widetilde{\iota}_2$ are inclusions into the first and second factors of $i_{C_2}^* \BPCnplusone$ respectively.  The $\ti^{C_{2^{n+1}}}$-generators are defined to be the coefficients of the power series relating $\widetilde{x}_1$ and $\widetilde{x}_2$: 
\begin{equation}
\widetilde{x}_2 = \widetilde{x}_1 + {\sum_{i \geq 1}}^{\widetilde{F}_2} \ti^{C_{2^{n+1}}}\widetilde{x}_1^{2^i}. \label{eq:tiC2nplusoneDef}
\end{equation}
After taking the norm $N_{C_2}^{C_4}(-)$ and applying the norm-restriction adjunction to obtain the maps 
\begin{multline*}
N_{C_2}^{C_4}(\widetilde{\iota}_1), N_{C_2}^{C_4}(\widetilde{\iota}_2): \BPCfour \longrightarrow N_{C_2}^{C_4} i_{C_2}^* \BPCnplusone \longrightarrow i_{C_4}^* \BPCnplusone \\
\simeq \underbrace{\BPCfour \wedge \cdots \wedge \BPCfour}_{2^{n-1}},
\end{multline*}
the observation is that $N_{C_2}^{C_4}(\widetilde{\iota}_1)$ and $N_{C_2}^{C_4}(\widetilde{\iota}_2)$ are inclusions into the first and second factors of $i_{C_4}^*\BPCnplusone$, respectively.  This implies that after taking $\Phi^{C_4}(-)$, the maps 
\[\Phi^{C_4}N_{C_2}^{C_4}(\widetilde{\iota}_1), \Phi^{C_4}N_{C_2}^{C_4}(\widetilde{\iota}_2): \HF \longrightarrow i_e^* N_e^{C_{2^{n-1}}} \HF\]
are again $\iota_1$ and $\iota_2$, respectively.  It follows that $\Phi^{C_4}N_{C_2}^{C_4}(\widetilde{x}_1) = a_1$ and $\Phi^{C_4}N_{C_2}^{C_4}(\widetilde{x}_2) = a_2$.  Since applying $\Phi^{C_4}N_{C_2}^{C_4}(-)$ is a ring homomorphism \cite[Proposition~2.59]{HHR}, applying $\Phi^{C_4}N_{C_2}^{C_4}(-)$ to (\ref{eq:tiC2nplusoneDef}) shows that the power series relating $a_1$ and $a_2$ is
\begin{equation}
a_2 = f(a_1) = a_1 + {\sum_{i \geq 1}}^{F_2} \Phi^{C_4} N_{C_2}^{C_4} (\ti^{C_{2^{n+1}}}) a_1^{2^i}.\label{eq:a1a2RelationTwo}
\end{equation}
Comparing the coefficients of the power series in (\ref{eq:a1a2RelationOne}) and (\ref{eq:a1a2RelationTwo}) implies the equality $\Phi^{C_2}(\ti^{C_{2^n}}) = \Phi^{C_4} N_{C_2}^{C_4} (\ti^{C_{2^{n+1}}})$, as desired. 
\end{proof}

\begin{prop} \label{prop:spectrumLevelEquivalence}
 For all $n \geq 1$ and $I \subseteq \N$, the following equivalence holds: 
 \[\EF[C_4] \wedge \BPCnplusone \langle I \rangle \simeq \EF[C_4] \wedge \Pb_{C_{2^{n+1}}/C_2}\BPCn\langle I \rangle.\]
\end{prop}
\begin{proof}
By ~\cref{lem:easyFact1}, the left hand side is equal to 
\begin{align*}
& \quad \, \,\Pb_{C_{2^{n+1}}/C_4}\left(\Phi^{C_4} \BPCnplusone \langle I \rangle \right) \\
&\simeq  \Pb_{C_{2^{n+1}}/C_4} \Phi^{C_4} \left(\BPCnplusone \smashover{S^0\left[C_{2^{n+1}}\cdot \ti^{C_{2^{n+1}}} \,|\, i \geq 1\right]} S^0\left[C_{2^{n+1}}\cdot \ti^{C_{2^{n+1}}}\,|\, i \in I\right]\right) \\ 
&\simeq  \Pb_{C_{2^{n+1}}/C_4} \left(N_e^{C_{2^{n-1}}} \HF \smashover{S^0\left[C_{2^{n-1}} \cdot \Phi^{C_4}N_{C_2}^{C_4}(\ti^{C_{2^{n+1}}}) \,|\, i \geq 1\right]} S^0\left[C_{2^{n-1}} \cdot \Phi^{C_4}N_{C_2}^{C_4}(\ti^{C_{2^{n+1}}}) \,|\, i \in I \right]\right).
\end{align*}
Here, for the last equivalence, we have used the fact that the geometric fixed points functor $\Phi^H(-)$ is a symmetric monoidal functor.  Furthermore, if $H$ is a normal subgroup of $G$ that contains $C_2$, then
\[\Phi^H S^0[G \cdot \bar{x}] \simeq S^0\left[G/H \cdot \Phi^H N_{C_2}^H (\bar{x})\right]\]
according to the properties of the functor (see \cite[Section~2]{HHR} and \cite[Lemma~4.3]{UllmanPaper}).

By ~\cref{lem:easyFact1}, the right hand side is equal to 
\begin{align*}
& \quad \,\, \Pb_{C_{2^{n+1}}/C_4}\left(\Phi^{C_4}\Pb_{C_{2^{n+1}}/C_2}(\BPCn\langle I \rangle)\right) \\
&\simeq  \Pb_{C_{2^{n+1}}/C_4}\left(\Phi^{C_2} \BPCn\langle I \rangle \right)\\
&\simeq  \Pb_{C_{2^{n+1}}/C_4}\Phi^{C_2}\left(\BPCn \smashover{S^0\left[C_{2^{n}}\cdot \ti^{C_{2^{n}}} \,|\, i \geq 1\right]} S^0\left[C_{2^{n}}\cdot \ti^{C_{2^{n}}}\,|\, i \in I\right]\right) \\
&\simeq  \Pb_{C_{2^{n+1}}/C_4} \left(N_e^{C_{2^{n-1}}}\HF \smashover{S^0\left[C_{2^{n-1}}\cdot \Phi^{C_2}(\ti^{C_{2^{n}}}) \,|\, i \geq 1\right]} S^0\left[C_{2^{n-1}}\cdot \Phi^{C_2}(\ti^{C_{2^{n}}})\,|\, i \in I\right]\right). 
\end{align*}
The equivalence now follows from ~\cref{lem:tiGeneratorRelations} and the fact that the maps 
\begin{align*}
S^0\left[C_{2^{n-1}} \cdot \Phi^{C_4}N_{C_2}^{C_4}(\ti^{C_{2^{n+1}}}) \,|\, i \geq 1\right] &\longrightarrow S^0\left[C_{2^{n-1}} \cdot \Phi^{C_4}N_{C_2}^{C_4}(\ti^{C_{2^{n+1}}}) \,|\, i \in I \right] \\
S^0\left[C_{2^{n-1}}\cdot \Phi^{C_2}(\ti^{C_{2^{n}}}) \,|\, i \geq 1\right] &\longrightarrow S^0\left[C_{2^{n-1}}\cdot \Phi^{C_2}(\ti^{C_{2^{n}}})\,|\, i \in I\right]
\end{align*}
are the same quotient maps via the identification $\Phi^{C_4} N_{C_2}^{C_4}(\ti^{C_{2^{n+1}}}) \leftrightsquigarrow \Phi^{C_2}(\ti^{C_{2^n}})$. 
\end{proof}

To proceed further, we will briefly recall the dual slice towers of $\BPCnplusone \langle I \rangle$ and $\BPCn \langle I \rangle$, following the construction in \cite[Section~6]{HHR}.  For more details on the slice towers of $\MUCn$ and its quotients, see \cite[Section~6]{HHR} and \cite[Section~2]{BHLSZ}.  
 
Let 
\[S^0\left[C_{2^{n+1}}\cdot \ti^{C_{2^{n+1}}} \,|\, i \geq 1\right] \longrightarrow \BPCnplusone\]
be the refinement of $\BPCnplusone$ with respect to the generators $\ti^{C_{2^{n+1}}}$, and let $A = S^0\left[C_{2^{n+1}}\cdot \ti^{C_{2^{n+1}}} \,|\, i \in I\right]$.  Then $A$ is a wedge summand of $S^0\left[C_{2^{n+1}}\cdot \ti^{C_{2^{n+1}}} \,|\, i \geq 1\right]$.  Applying $(-) \wedge_{S^0[C_{2^{n+1}}\cdot \ti^{C_{2^{n+1}}} \,|\, i \geq 1]} A$ to the refinement above produces the refinement 
\[A \longrightarrow \BPCnplusone \langle I \rangle\]
for $\BPCn \langle I \rangle$. 

The spectrum $\BPCnplusone \langle I \rangle$ is a $S^0\left[C_{2^{n+1}}\cdot \ti^{C_{2^{n+1}}} \,|\, i \geq 1\right]$-module, and we can further view it as an $A$-module via the associative algebra map 
\[A \longrightarrow S^0\left[C_{2^{n+1}}\cdot \ti^{C_{2^{n+1}}} \,|\, i \geq 1\right].\]
Define $M_{\geq d} \subset A$ to be the monomial ideal consisting of all the spheres of dimension $\geq d$.  These monomial ideals produces an increasing filtration
\[\cdots \hookrightarrow M_{\geq d+1} \hookrightarrow M_{\geq d} \hookrightarrow M_{\geq d-1} \hookrightarrow \cdots\]
on $A$.  Define 
\begin{equation}\label{eq:DualSlicesBPCnplusone}
P_d\left(\BPCnplusone\langle I \rangle\right):= \BPCnplusone\langle I \rangle \wedge_{A} M_{\geq d}.
\end{equation}
Then the tower 
\[\cdots \longrightarrow P_{d+1}\left(\BPCnplusone\langle I \rangle\right) \longrightarrow P_{d}\left(\BPCnplusone\langle I \rangle\right) \longrightarrow P_{d-1}\left(\BPCnplusone\langle I \rangle\right) \longrightarrow \cdots\]
is the dual slice tower of $\BPCnplusone\langle I \rangle$.  

Smashing the dual slice tower with $\EF[C_4]$ produces the localized dual slice tower $\EF[C_4] \wedge P_\bullet \left(\BPCnplusone\langle I \rangle\right)$.  By \cref{lem:easyFact1}, the following equivalence holds: 
\begin{align}
& \quad \,\,\EF[C_4] \wedge P_d\left(\BPCnplusone\langle I \rangle\right) \nonumber \\
&\simeq \EF[C_4] \wedge \left(\BPCnplusone\langle I \rangle \wedge_A M_{\geq d}\right) \nonumber \\
&\simeq \Pb_{C_{2^{n+1}}/C_4}\left(\Phi^{C_4}(\BPCnplusone\langle I \rangle) \smashover{\Phi^{C_4}(A)} \Phi^{C_4}(M_{\geq d})\right). \label{eq:TowerForCnplusone}
\end{align}

The construction of the dual slice tower of $\BPCn\langle I \rangle$ follows a  similar method.  Let $B \to \BPCn\langle I \rangle$ be the refinement of $\BPCn\langle I \rangle$ with respect to the generators $\bar{t}_i^{C_{2^n}}$, where 
\[B = S^0\left[C_{2^n}\cdot \ti^{C_{2^n}} \,|\, i \in I\right].\]
Define $N_{\geq d} \subset B$ to be the monomial ideal consisting of all the spheres of dimension $\geq d$ and let 
\[P_d\left(\BPCn\langle I \rangle\right) := \BPCn\langle I \rangle \wedge_B N_{\geq d}.\]
Then the tower 
\[\cdots \longrightarrow P_{d+1}\left(\BPCn\langle I \rangle\right) \longrightarrow P_{d}\left(\BPCn\langle I \rangle\right) \longrightarrow P_{d-1}\left(\BPCn\langle I \rangle\right) \longrightarrow \cdots\]
is the dual slice tower of $\BPCn\langle I \rangle$.  Smashing this dual slice tower with $\EF[C_2]$ produces the localized dual slice tower $\EF[C_2] \wedge P_\bullet (\BPCn \langle I \rangle)$.  By \cref{lem:easyFact1}, we have the following equivalence: 
\begin{align}
\EF[C_2] \wedge P_d\left(\BPCn \langle I \rangle \right) &\simeq   \EF[C_2] \wedge\left( \BPCn\langle I \rangle \wedge_B N_{\geq d}\right) \nonumber \\
&\simeq   \Pb_{C_{2^{n}}/C_2} \left( \Phi^{C_2}(\BPCn\langle I \rangle) \smashover{\Phi^{C_2}(B)} \Phi^{C_2}(N_{\geq d}) \right) \label{eq:TowerForCn}
\end{align}

\begin{prop}\label{prop:LocalizedDualTowerEquivalence1}
We have the following equivalence of towers:
\[\EF[C_4] \wedge P_\bullet \left(\BPCnplusone\langle I \rangle\right) \simeq \EF[C_4] \wedge P_\bullet \left(\mathcal{P}_{C_{2^{n+1}}/C_2}^*\BPCn\langle I \rangle\right).\]
\end{prop}
\begin{proof}
For any $N \subseteq G$ a normal subgroup and $H \subseteq G$ a subgroup, there is a natural equivalence (see \cite[Section~4.2]{HillPrimer})
\[\Phi^{N}(G_+ \wedge_H S^{k\rho_H}) \simeq \left\{\begin{array}{ll}
* & \text{if } N \not \subseteq H, \\
(G/N)_+ \wedge_{H/N} S^{k\rho_{H/N}} & \text{if } N \subseteq H. \end{array} \right. \]
This fact, combined with (\ref{eq:TowerForCnplusone}), shows that for all $d$ and $0 \leq i \leq 3$, the localized dual slice sections
\[\EF[C_4] \wedge P_{4d-i} \left(\BPCnplusone\langle I \rangle\right) \]
are equivalent to each other.  Similarly, by (\ref{eq:TowerForCn}) and \cref{prop:LocalizedDualTowerEquivalence2}, 
\[\EF[C_4] \wedge P_{4d-i} \left(\mathcal{P}_{C_{2^{n+1}}/C_2}^*\BPCn\langle I \rangle\right)\]
are also equivalent to each other for all $d$ and $0 \leq i \leq 3$.  Therefore, in order to prove the desired equivalence of towers, it suffices to show that for all $d$, we have the equivalence
\[\EF[C_4] \wedge P_{4d} \left(\BPCnplusone\langle I \rangle\right) \simeq \EF[C_4] \wedge P_{4d} \left(\Pb_{C_{2^{n+1}}/C_2}\BPCn\langle I \rangle\right),\]
and that the maps 
\begin{align*}
\EF[C_4] \wedge P_{4d+4} \left(\BPCnplusone\langle I \rangle\right) &\longrightarrow \EF[C_4] \wedge P_{4d} \left(\BPCnplusone\langle I \rangle\right), \\
\EF[C_4] \wedge P_{4d+4} \left(\Pb_{C_{2^{n+1}}/C_2}\BPCn\langle I \rangle\right) &\longrightarrow \EF[C_4] \wedge P_{4d} \left(\Pb_{C_{2^{n+1}}/C_2}\BPCn\langle I \rangle\right)
\end{align*}
are the same.  By (\ref{eq:TowerForCnplusone}), \cref{prop:LocalizedDualTowerEquivalence2}, and (\ref{eq:TowerForCn}), this is equivalent to showing that 
\begin{equation}\label{eq:geomFixPointEquivalence}
\Phi^{C_4}(\BPCnplusone\langle I \rangle) \smashover{\Phi^{C_4}(A)} \Phi^{C_4}(M_{\geq 4d}) \simeq \Phi^{C_2}(\BPCn\langle I \rangle) \smashover{\Phi^{C_2}(B)} \Phi^{C_2}(N_{\geq 2d}) 
\end{equation}
and that the two maps 
\begin{equation}\label{map:geomFixedPointMap1}
\Phi^{C_4}(\BPCnplusone\langle I \rangle) \smashover{\Phi^{C_4}(A)} \Phi^{C_4}(M_{\geq 4d+4}) \longrightarrow \Phi^{C_4}(\BPCnplusone\langle I \rangle) \smashover{\Phi^{C_4}(A)} \Phi^{C_4}(M_{\geq 4d})  
\end{equation}
\begin{equation}\label{map:geomFixedPointMap2}
\Phi^{C_2}(\BPCn\langle I \rangle) \smashover{\Phi^{C_2}(B)} \Phi^{C_2}(N_{\geq 2d+2}) \longrightarrow \Phi^{C_2}(\BPCn\langle I \rangle) \smashover{\Phi^{C_2}(B)} \Phi^{C_2}(N_{\geq 2d}) 
\end{equation}
are the same. 

Note that since 
\begin{align*}
\Phi^{C_4}(A) &= S^0\left[C_{2^{n-1}} \cdot \Phi^{C_4}N_{C_2}^{C_4}(\ti^{C_{2^{n+1}}}) \, |\, i \in I\right], \\
\Phi^{C_2}(B) &=S^0\left[C_{2^{n-1}} \cdot \Phi^{C_2}(\ti^{C_{2^n}}) \,|\, i \in I \right],
\end{align*}
they are isomorphic wedges of slice cells under the identification 
\[\Phi^{C_4} N_{C_2}^{C_4}(\ti^{C_{2^{n+1}}}) \leftrightsquigarrow \Phi^{C_2}(\ti^{C_{2^n}}).\]  
Also, under this identification, the quotient maps 
\[\Phi^{C_4}(A) \longrightarrow \Phi^{C_4}(M_{\geq 4d})\]
and 
\[\Phi^{C_2}(B) \longrightarrow \Phi^{C_2}(N_{\geq 2d})\]
are the same.  Furthermore, by Lemma~\ref{lem:tiGeneratorRelations}, and the proof of Proposition~\ref{prop:spectrumLevelEquivalence}, the maps 
\begin{align*}
\Phi^{C_4}(A) &\longrightarrow \Phi^{C_4}(\BPCnplusone\langle I \rangle), \\
\Phi^{C_2}(B) &\longrightarrow \Phi^{C_2}(\BPCn\langle I \rangle)
\end{align*}
are also the same.  It follows that equivalence~(\ref{eq:geomFixPointEquivalence}) holds.  

To prove that the maps (\ref{map:geomFixedPointMap1}) and (\ref{map:geomFixedPointMap2}) are the same, it suffices to show that the inclusion maps
\begin{align*}
\Phi^{C_4}(M_{\geq 4d+4}) &\longrightarrow \Phi^{C_4}(M_{\geq 4d}),\\
\Phi^{C_2}(N_{\geq 2d+2}) &\longrightarrow \Phi^{C_2}(N_{\geq 2d})
\end{align*}
are the same.  This follows directly from the same identification 
\[\Phi^{C_4} N_{C_2}^{C_4}(\ti^{C_{2^{n+1}}}) \leftrightsquigarrow \Phi^{C_2}(\ti^{C_{2^n}})\]  
earlier. 
\end{proof}

\begin{prop}\label{prop:LocalizedDualTowerEquivalence2}
We have the following equivalence of towers:
\[\EF[C_4] \wedge P_\bullet \left(\Pb_{C_{2^{n+1}}/C_2}\BPCn \langle I \rangle \right) \simeq \Pb_{C_{2^{n+1}}/C_2}\left(\EF[C_2] \wedge \D P_\bullet (\BPCn \langle I \rangle)\right).\]
\end{prop}
\begin{proof}
By \cref{prop:pullBackSliceTower}, there is an equivalence of towers
\begin{equation}\label{eq:UllmanTowerEquivalence}
P_\bullet \left(\Pb_{C_{2^{n+1}}/C_2}\BPCn \langle I \rangle \right) \simeq \Pb_{C_{2^{n+1}}/C_2} \left( \D P_\bullet (\BPCn \langle I \rangle ) \right). 
\end{equation}
More precisely, for all $d \geq 1$, we have the equivalences 
\begin{align*}
P_{2d} \left(\Pb_{C_{2^{n+1}}/C_2}\BPCn\langle I \rangle \right) & \simeq P_{2d-1} \left(\Pb_{C_{2^{n+1}}/C_2} \BPCn\langle I \rangle \right) \\
& \simeq \Pb_{C_{2^{n+1}}/C_2}\left(P_d(\BPCn\langle I \rangle)\right).
\end{align*}
Therefore, by (\ref{eq:UllmanTowerEquivalence}), the left hand side of the equivalence in the proposition statement is equivalent to 
\begin{align}
&\quad \,\, \EF[C_4] \wedge \Pb_{C_{2^{n+1}}/C_2} \left( \D P_{\bullet}(\BPCn\langle I \rangle) \right) \nonumber \\ 
&\simeq \Pb_{C_{2^{n+1}}/C_4}\left(\Phi^{C_4} \Pb_{C_{2^{n+1}}/C_2} \D P_{\bullet}(\BPCn\langle I \rangle) \right) \nonumber \,\,\, (\text{\cref{lem:easyFact1}})\\
&\simeq \Pb_{C_{2^{n+1}}/C_4} \left(\Phi^{C_2} \D P_{\bullet}(\BPCn\langle I \rangle \right). \label{eq:IdentificationPullback}
\end{align}
By \cref{lem:easyFact1}, the right hand side is equivalent to
\[\Pb_{C_{2^{n+1}}/C_2} \left(\Pb_{C_{2^{n}}/C_2} \Phi^{C_2} \D P_{\bullet}(\BPCn\langle I \rangle\right) \simeq \Pb_{C_{2^{n+1}}/C_4} \left(\Phi^{C_2} \D P_{\bullet}(\BPCn\langle I \rangle \right). \]
This proves the desired equivalence.  
\end{proof}

\begin{proof}[Proof of \cref{thm:DualTowerEquivalenceMain}]
The case $k =1$ is an immediate consequence of \cref{prop:LocalizedDualTowerEquivalence1} and \cref{prop:LocalizedDualTowerEquivalence2}:
\begin{align*}
\EF[C_{4}] \wedge P_\bullet\left(\BPCnplusone \langle I \rangle \right) & \simeq \EF[C_4] \wedge P_\bullet \left(\mathcal{P}_{C_{2^{n+1}}/C_2}^*\BPCn\langle I \rangle\right) \\
&\simeq \Pb_{C_{2^{n+1}}/C_2}\left(\EF[C_{2}] \wedge \D P_\bullet (\BPCn \langle I \rangle) \right) \\
&\simeq \Pb_{C_{2^{n+1}}/C_2} \D \left(\EF[C_{2}] \wedge P_\bullet (\BPCn \langle I \rangle) \right). 
\end{align*}

For the general case when $1 \leq k \leq n$, the left hand side of the equivalence is
\begin{align*}
& \quad \,\, \EF[C_{2^{k+1}}] \wedge P_\bullet\left(\BPCnplusone \langle I \rangle \right) \\
&\simeq  \EF[C_{2^{k+1}}] \wedge \left(\EF[C_4] \wedge P_\bullet (\BPCnplusone\langle I \rangle ) \right) \\
&\simeq \EF[C_{2^{k+1}}] \wedge \Pb_{C_{2^{n+1}}/C_2} \D \left(\EF[C_{2}] \wedge P_\bullet (\BPCn \langle I \rangle) \right) \,\,\, \text{($k=1$ case)}\\
&\simeq  \EF[C_{2^{k+1}}] \wedge \Pb_{C_{2^{n+1}}/C_2}\left(\EF[C_{2}] \wedge \D P_\bullet (\BPCn \langle I \rangle) \right) \\
&\simeq  \Pb_{C_{2^{n+1}}/C_{2^{k+1}}} \Phi^{C_{2^{k+1}}}\left(\Pb_{C_{2^{n+1}}/C_2}\left(\EF[C_{2}] \wedge \D P_\bullet (\BPCn \langle I \rangle) \right) \right)\,\,\, \text{(\cref{lem:easyFact1})}\\
&\simeq  \Pb_{C_{2^{n+1}}/C_{2^{k+1}}} \Phi^{C_{2^{k}}}\left(\EF[C_{2}] \wedge \D P_\bullet (\BPCn \langle I \rangle) \right) \\
&\simeq  \Pb_{C_{2^{n+1}}/C_{2^{k+1}}} \Phi^{C_{2^{k}}}\left( \D P_\bullet (\BPCn \langle I \rangle) \right).
\end{align*}
The right hand side of the equivalence is 
\begin{align*}
& \quad \,\, \Pb_{C_{2^{n+1}}/C_2}\D \left( \EF[C_{2^k}] \wedge P_\bullet (\BPCn \langle I \rangle) \right) \\ 
&\simeq \Pb_{C_{2^{n+1}}/C_2}\left(\EF[C_{2^k}] \wedge \D P_\bullet (\BPCn \langle I \rangle) \right) \\ 
&\simeq \Pb_{C_{2^{n+1}}/C_2} \left(\Pb_{C_{2^n}/C_{2^k}} \Phi^{C_{2^k}}\left( \D P_\bullet(\BPCn \langle I \rangle)\right) \right) \,\,\, \text{(\cref{lem:easyFact1})}\\
&\simeq   \Pb_{C_{2^{n+1}}/C_{2^{k+1}}} \Phi^{C_{2^{k}}}\left( \D P_\bullet (\BPCn \langle I \rangle) \right).
\end{align*}
This establishes the desired equivalence.
\end{proof}

\section{Correspondence formulas} \label{sec:CorrespondenceFormulas}

By \cref{thm:ShearingGeneralTheory}, the equivalence of localized dual slice towers proven in \cref{thm:DualTowerEquivalenceMain} induces a shearing isomorphism of the corresponding localized slice spectral sequences.  In this section, we will establish correspondence formulas between classes in these localized slice spectral sequences under the shearing isomorphism. 

In order to establish our formulas, we will first focus on the equivalence 
\begin{equation}\label{eq:correspondenceFormulaTower}
\EF[C_4] \wedge P_\bullet\left(\BPCnplusone \right) \simeq \Pb_{C_{2^{n+1}}/C_2}\D \left( \EF[C_2] \wedge P_\bullet (\BPCn) \right).
\end{equation}
Once we have established correspondence formulas for classes on the $\mathcal{E}_2$-pages of $\EF[C_4] \wedge \SliceSS\left(\BPCnplusone \right)$ and $\EF[C_2] \wedge \SliceSS (\BPCn )$, the correspondence formulas for all the other equivalences in \cref{thm:DualTowerEquivalenceMain} will follow.  

To start our analysis, we will first recall some notations.  When working 2-locally, the representation spheres associated to the 2-dimensional real $\Cn$-representations corresponding to rotations by $\left(\frac{2\pi \cdot k}{2^n}\right)$ and $\left(\frac{2\pi \cdot k'}{2^n}\right)$ are equivalent whenever $k$ and $k'$ have the same 2-adic valuation.  Therefore, the $RO(C_{2^n})$-graded homotopy groups are graded by the representations $\{1, \sigma, \lambda_1, \ldots, \lambda_{n-1}\}$, and the $RO(C_{2^{n+1}})$-graded homotopy groups are graded by the representations $\{1, \sigma, \lambda_1, \ldots, \lambda_n\}$.  Here, 1 is the trivial representation, $\sigma$ is the sign representation, and $\lambda_i$ is the 2-dimensional real representation that corresponds to rotation by $\left(\frac{\pi}{2^i} \right)$.

We will denote the $C_{2^{n}}$- and the $C_{2^{n+1}}$-equivariant constant Mackey functors by $\Z_{C_{2^n}}$ and $\Z_{C_{2^{n+1}}}$, respectively.

The slice associated graded for the tower $\EF[C_4] \wedge P_\bullet\left(\BPCnplusone \right)$ is 
\[\EF[C_4] \wedge H\Z_{C_{2^{n+1}}}\left[C_{2^{n+1}}\cdot \ti^{C_{2^{n+1}}} \,|\, i \geq 1 \right] \simeq  a_{\lambda_{n-1}}^{-1}H\Z_{C_{2^{n+1}}}\left[C_{2^{n}}\cdot N_{C_2}^{C_4}(\ti^{C_{2^{n+1}}}) \,|\, i \geq 1 \right],\]
and the slice associated graded for the tower $\EF[C_2] \wedge P_\bullet (\BPCn)$ is 
\[\EF[C_2] \wedge H\Z_{C_{2^{n}}}\left[C_{2^{n}}\cdot \ti^{C_{2^{n}}} \,|\, i \geq 1 \right] \simeq a_{\lambda_{n-1}}^{-1}H\Z_{C_{2^{n}}} \left[C_{2^{n}}\cdot \ti^{C_{2^{n}}} \,|\, i \geq 1 \right].\]
For every $V \in RO(\Cnplusone)$ and $C_2 \subseteq H \subseteq \Cnplusone$, ~(\ref{eq:correspondenceFormulaTower}) induces an isomorphism 
\begin{align}\label{eq:correspondenceFormulaHtpyIsom}
& \,\, \pi_V^H \left(a_{\lambda_{n-1}}^{-1}H\Z_{C_{2^{n+1}}}\left[C_{2^{n}}\cdot N_{C_2}^{C_4}(\ti^{C_{2^{n+1}}}) \,|\, i \geq 1\right] \right) \nonumber \\
\cong & \,\, \pi_{V^{C_2}}^{H/C_2} \left(a_{\lambda_{n-1}}^{-1}H\Z_{C_{2^{n}}} \left[C_{2^{n}}\cdot \ti^{C_{2^{n}}} \,|\, i \geq 1 \right] \right)
\end{align}
by \cref{lem:PullbackHomotopyGroups}.  There are distinguished generators for the homotopy groups on each side of the above isomorphism ($a_V$, $u_V$, $\bar{t}_i^{C_{2^n}}$, $N_{C_2}^{C_4}(\ti^{C_{2^{n+1}}})$).  Our first goal is to prove correspondence formulas for these generators.    

\begin{prop}\label{prop:aVCorrespondenceFormula}
Suppose $V \in RO(C_{2^{n+1}}/C_2)$ is of the form 
\[V = c_0 \sigma + c_1 \lambda_1 + \cdots + c_{n-1} \lambda_{n-1},\] where $c_i \geq 0$ for all $0 \leq i \leq n-1$.  Equivalence~(\ref{eq:correspondenceFormulaTower}) and isomorphism~(\ref{eq:correspondenceFormulaHtpyIsom}) induce the correspondence $a_V \leftrightsquigarrow a_V$.  Here, $V$ is also viewed as an element of $RO(C_{2^{n+1}})$ through pullback along the quotient map $C_{2^{n+1}} \to C_{2^{n+1}}/C_2$.
\end{prop}
\begin{proof}
By definition, on the left hand side of isomorphism~(\ref{eq:correspondenceFormulaHtpyIsom}), the special class ${a_V \in \pi_{-V}^{C_{2^{n+1}}} (\EF[C_4] \wedge H\Z_{C_{2^{n+1}}})}$ is defined to be the composite 
\[S^{-V} \xrightarrow{a_V} S^0 \longrightarrow \EF[C_4] \wedge H\Z_{C_{2^{n+1}}} \simeq \Pb_{C_{2^{n+1}}/C_2}(\EF[C_2] \wedge H\Z_{C_{2^n}}),\]
where the first map is the Euler class $a_V \in a_{-V} S^0$ that corresponds to the inclusion map $S^0 \to S^V$, and the last equivalence can be deduced from (\ref{eq:correspondenceFormulaTower}).  By adjunction between the functors $\Pb_{C_{2^{n+1}}/C_2}$ and $\Phi^{C_2}$, the composition above corresponds to the composition
\[\Phi^{C_2}(S^{-V}) \xrightarrow{\Phi^{C_2}(a_V)} \Phi^{C_2}(S^0) \longrightarrow \EF[C_2] \wedge H\Z_{C_{2^n}}.\]

By our assumption on $V$, we have the equality $V^{C_2} = V$.  Therefore, $\Phi^{C_2}(S^{-V}) = S^{-V}$ and ${\Phi^{C_2}(a_V) = a_V \in \pi_{-V}^{C_{2^n}} S^0}$.  It follows that the composition map is the class ${a_V \in \pi_{-V}^{\Cn} (\EF[C_2] \wedge H\Z_{\Cn})}$.  
\end{proof}

\begin{prop}\label{prop:uVCorrespondenceFormula}
Suppose $V \in RO(\Cnplusone/C_2)$ is of the form 
\[c_0 (2\sigma) + c_1 \lambda_1 + \cdots + c_{n-1} \lambda_{n-1},\]
where $c_i \geq 0$ for all $0 \leq i \leq n-1$.  Equivalence~(\ref{eq:correspondenceFormulaTower}) and isomorphism~(\ref{eq:correspondenceFormulaHtpyIsom}) induce the correspondence $u_V \leftrightsquigarrow u_V$, where $V$ is also viewed as an element of $RO(C_{2^{n+1}})$ through pullback along the quotient map $C_{2^{n+1}} \to C_{2^{n+1}}/C_2$.
\end{prop}
\begin{proof}
Note that by our assumption, $V$ is orientable, and the orientation class $u_V$ is defined to be the class in ${\pi_{|V|-V}^{\Cnplusone} (\EF[C_4] \wedge H\Z_{\Cnplusone})}$ that represents the composite map 
\[S^{|V|-V} \xrightarrow{u_V} H\Z_{C_{2^{n+1}}} \longrightarrow \EF[C_4] \wedge H\Z_{C_{2^{n+1}}} \simeq \Pb_{C_{2^{n+1}}/C_2}(\EF[C_2] \wedge H\Z_{C_{2^n}}).\]
Here, the first map $u_V$ is the (preferred) generator for $\pi_{|V| - V}^{\Cnplusone} H\Z_{\Cnplusone} \cong \mathbb{Z}$, and the second map is the quotient map 
\[\mathbb{Z} \cong \pi_{|V| - V}^{C_{2^{n+1}}} H\Z_{C_{2^{n+1}}} \longrightarrow \pi_{|V| - V}^{C_{2^{n+1}}} (a_{\lambda_{n-1}}^{-1} H\Z_{C_{2^{n+1}}}) \cong \mathbb{Z}/2^{n}.\]

By adjunction between the functors $\Pb_{C_{2^{n+1}}/C_2}$ and $\Phi^{C_2}$, the composite map $u_V$ corresponds to the map 
\[S^{|V| - V} \simeq \Phi^{C_2}(S^{|V| - V}) \longrightarrow \EF[C_2] \wedge H\Z_{\Cn} \simeq a_{\lambda_{n-1}}^{-1} H\Z_{\Cn}.\]
This represents the generator in $\pi_{|V| - V}^{\Cn} (a_{\lambda_{n-1}}^{-1} H\Z_{\Cn}) \cong \mathbb{Z}/2^{n}$.  Therefore, this map must be in the image of a generator $u_V$ for $\pi_{|V| - V}^{\Cn} H\Z_{\Cn} \cong \mathbb{Z}$ under the composition 
\[S^{|V| - V} \longrightarrow H\Z_{\Cn} \longrightarrow \EF[C_2] \wedge H\Z_{\Cn}.\]
This shows that $u_V \leftrightsquigarrow u_V$.  
\end{proof}

\begin{prop}\label{prop:tiCorrespondenceFormula}
For all $1 \leq j \leq n$, Equivalence~(\ref{eq:correspondenceFormulaTower}) and isomorphism~(\ref{eq:correspondenceFormulaHtpyIsom}) induce the correspondence 
\[ N_{C_2}^{C_{2^{j+1}}}(\ti^{C_{2^{n+1}}})a_{\lambda_j}^{2^{j-1}(2^i-1)} \leftrightsquigarrow N_{C_2}^{C_{2^j}} (\ti^{C_{2^n}}) \]
\end{prop}
\begin{proof}
We will first prove that when $j=1$, we have the correspondence 
\[ N_{C_2}^{C_4} (\ti^{C_{2^{n+1}}}) a_{\lambda_1}^{2^i-1} \leftrightsquigarrow  \ti^{C_{2^n}}.\]
The general case will follow immediately by applying the norm.  

By \cref{lem:tiGeneratorRelations}, we have the equality 
\[\Phi^{C_2}(\ti^{C_{2^n}}) = \Phi^{C_4} N_{C_2}^{C_4} (\ti^{C_{2^{n+1}}}).\]
For the map 
\[\ti^{\Cn}: S^{(2^i-1)(1+\sigma)} \longrightarrow i_{C_2}^* \BPCn,\]
the $C_2$-geometric fixed points $\Phi^{C_2}(\ti^{C_{2^n}})$ is represented by the composite map $\ti^{\Cn}a_\sigma^{2^i-1}$: 
\[S^{2^i-1} \xrightarrow{a_\sigma^{2^i-1}} S^{(2^i-1)(1+\sigma)} \xrightarrow{\ti^{\Cn}} i_{C_2}^*\BPCn \longrightarrow \EF[C_2] \wedge i_{C_2}^*\BPCn.\]
Similarly, for the map 
\[N_{C_2}^{C_4} (\ti^{C_{2^{n+1}}}): S^{(2^i-1)(1 + \sigma + \lambda_1)} \longrightarrow i_{C_4}^* \BPCnplusone,\]
the $C_4$-geometric fixed points $\Phi^{C_4} N_{C_2}^{C_4} (\ti^{C_{2^{n+1}}})$ is represented by the composite map $N_{C_2}^{C_4} (\ti^{C_{2^{n+1}}}) a_{\bar{\rho}_4}^{2^i-1}:$
\[S^{2^i-1} \xrightarrow{a_{\bar{\rho}_4}^{2^i-1}} S^{(2^i-1)(1 + \sigma + \lambda_1)} \xrightarrow{N_{C_2}^{C_4} (\ti^{C_{2^{n+1}}})} i_{C_4}^* \BPCnplusone \longrightarrow \EF[C_4] \wedge i_{C_4}^* \BPCnplusone.\]
Therefore, this implies that we have the correspondence 
\[\ti^{C_{2^n}} a_\sigma^{2^i-1} \leftrightsquigarrow N_{C_2}^{C_4} (\ti^{C_{2^{n+1}}}) a_{\bar{\rho}_4}^{2^i-1} = N_{C_2}^{C_4} (\ti^{C_{2^{n+1}}}) a_\sigma^{2^i-1}a_{\lambda_1}^{2^i-1}.\]
Since we have already established the correspondence $a_\sigma \leftrightsquigarrow a_\sigma$ in \cref{prop:aVCorrespondenceFormula}, we obtain the correspondence   
\[\ti^{C_{2^n}} \leftrightsquigarrow N_{C_2}^{C_4} (\ti^{C_{2^{n+1}}}) a_{\lambda_1}^{2^i-1}. \]
Applying the norm, we deduce that for $1 \leq j \leq n$, we have the correspondence
\[N_{C_2}^{C_{2^j}} (\ti^{C_{2^n}}) \leftrightsquigarrow N_{C_4}^{C_{2^{j+1}}}\left( N_{C_2}^{C_4}(\ti^{\Cnplusone}) a_{\lambda_1}^{2^i-1} \right) = N_{C_2}^{C_{2^{j+1}}}(\ti^{\Cnplusone})a_{\lambda_j}^{2^{j-1}(2^i-1)}.\]
\end{proof}

We are now ready to prove correspondence formulas for the general case of the equivalence in \cref{thm:DualTowerEquivalenceMain}, which states that for $1 \leq k \leq n$, we have the equivalence
\begin{equation}\label{eq:corresopndenceFormulaGeneralCase1}
\EF[C_{2^{k+1}}] \wedge P_\bullet\left(\BPCnplusone \langle I \rangle \right) \simeq \Pb_{C_{2^{n+1}}/C_{2^k}}\D^k\left(\EF[C_{2}] \wedge P_\bullet (\BPCnminuskplusone \langle I \rangle) \right).
\end{equation}
The slice associated graded for the tower $\EF[C_{2^{k+1}}] \wedge P_\bullet\left(\BPCnplusone \langle I \rangle \right)$ is 
\begin{align*}
& \,\,\EF[\Ckplusone] \wedge H\Z_{\Cnplusone}\left[\Cnplusone \cdot \ti^{\Cnplusone} \,|\, i \in I\right]\\
\simeq & \,\, a_{\lambda_{n-k}}^{-1}H\Z_{\Cnplusone} \left[\Cnminuskplusone \cdot N_{C_2}^{\Ckplusone}(\ti^{\Cnplusone}) \,|\, i \in I \right],
\end{align*}
and the slice associated graded for the tower $\EF[C_{2}] \wedge P_\bullet (\BPCnminuskplusone \langle I \rangle)$ is 
\begin{align*}
& \,\, \EF[C_2] \wedge H\Z_{\Cnminuskplusone}\left[\Cnminuskplusone \cdot \ti^{\Cnminuskplusone} \,|\, i \in I \right]\\ 
\simeq & \,\, a_{\lambda_{n-k}}^{-1} H\Z_{\Cnminuskplusone} \left[\Cnminuskplusone \cdot \ti^{\Cnminuskplusone} \,|\, i \in I \right].
\end{align*}
For every $V \in RO(\Cnplusone)$ and $\Ck \subseteq H \subseteq \Cnplusone$, ~(\ref{eq:corresopndenceFormulaGeneralCase1}) induces an isomorphism 
\begin{align}\label{eq:correspondenceGeneralHtpIsom}
& \,\, \pi_V^{H} \left(a_{\lambda_{n-k}}^{-1}H\Z_{\Cnplusone} \left[\Cnminuskplusone \cdot N_{C_2}^{\Ckplusone}(\ti^{\Cnplusone}) \,|\, i \in I \right] \right) \nonumber \\
\cong & \,\, \Mpi_{V^{\Ck}}^{H/\Ck} \left(a_{\lambda_{n-k}}^{-1} H\Z_{\Cnminuskplusone} \left[\Cnminuskplusone \cdot \ti^{\Cnminuskplusone} \,|\, i \in I \right] \right)
\end{align}
by \cref{lem:PullbackHomotopyGroups}.  

Recall that the reduced regular representation $\rhobar_{2^{n+1}} \in RO(\Cnplusone)$ can be decomposed as 
\[\rhobar_{2^{n+1}} = 2^{n-1}\lambda_n + 2^{n-2}\lambda_{n-1} + \cdots + \lambda_1 + \sigma .\]
\begin{df}\rm \label{df:RhobarNotation}
For $1 \leq k \leq n+1$, define $\rhobar_{2^{n+1}/2^{k}} \in RO(\Cnplusone)$ to be the representation 
\[\rhobar_{2^{n+1}/2^{k}} := 2^{n-k-1}\lambda_{n-k} + 2^{n-k-2}\lambda_{n-k-1}+ \ldots + \lambda_1 + \sigma.\]
\end{df}
 
\begin{thm}[Correspondence formulas]\label{thm:CorrespondenceFormula} 
Equivalence~(\ref{eq:corresopndenceFormulaGeneralCase1}) and isomorphism~(\ref{eq:correspondenceGeneralHtpIsom}) induce the following correspondences: 
\begin{enumerate}
\item Suppose $V \in RO(\Cnplusone/\Ck) = RO(\Cnminuskplusone)$ is of the form $V = c_0 \sigma + c_1 \lambda_1 + \cdots + c_{n-k} \lambda_{n-k}$, where $c_i \geq 0$ for all $0 \leq i \leq n-k$.  Then we have the correspondence 
\[a_V \leftrightsquigarrow a_V,\]
where $V$ is also viewed as an element of $RO(\Cnplusone)$ through pullback along the map $\Cnplusone \to \Cnplusone/\Ck$. 
\item Suppose $V \in RO(\Cnplusone/\Ck) = RO(\Cnminuskplusone)$ is of the form $V = c_0 (2\sigma) + c_1 \lambda_1 + \cdots + c_{n-k}\lambda_{n-k}$, where $c_i \geq 0$ for all $0 \leq i \leq n-k$.  Then we have the correspondence
\[u_V \leftrightsquigarrow u_V, \]
where as above $V$ is also viewed as an element of $RO(\Cnplusone)$ through pullback along the map $\Cnplusone \to \Cnplusone/\Ck$. 
\item For all $i \in I$ and $1 \leq j \leq n -k +1$, we have the correspondence 
\[N_{C_2}^{C_{2^{k+j}}}(\ti^{\Cnplusone}) \cdot \left(\frac{a_{\rhobar_{2^{k+j}}}}{a_{\rhobar_{2^{k+j}/2^k}}} \right)^{2^i-1} \leftrightsquigarrow N_{C_2}^{C_{2^j}}(\ti^{\Cnminuskplusone}).\]
\end{enumerate}
\end{thm}
\begin{proof}
Correspondences (1) and (2) follow immediately from \cref{prop:aVCorrespondenceFormula} and \cref{prop:uVCorrespondenceFormula}.  To prove correspondence (3), we will apply \cref{prop:tiCorrespondenceFormula} iteratively.  For $j =1$, we have the correspondences
\begin{align*}
\ti^{\Cnminuskplusone} & \leftrightsquigarrow N_{C_2}^{C_4}(\ti^{C_{2^{n-k+2}}})\cdot a_{\lambda_1}^{2^i-1}\\     
& \leftrightsquigarrow N_{C_2}^{C_8}(\ti^{C_{2^{n-k+3}}})\cdot a_{\lambda_2}^{2(2^i-1)}a_{\lambda_1}^{2^i-1}\\
& \,\,\,\,\, \vdots \\
& \leftrightsquigarrow N_{C_2}^{\Ckplusone}(\ti^{C_{2^{n+1}}})\cdot a_{\lambda_k}^{2^{k-1}(2^i-1)}a_{\lambda_{k-1}}^{2^{k-2}(2^i-1)} \cdots a_{\lambda_2}^{2(2^i-1)} a_{\lambda_1}^{(2^i-1)}
\end{align*}
Applying the norm, it follows that for $1 \leq j \leq n-k+1$, we have the correspondence 
\begin{align*}
N_{C_2}^{C_{2^j}}(\ti^{\Cnminuskplusone}) & \leftrightsquigarrow N_{\Ckplusone}^{C_{2^{k+j}}} \left( N_{C_2}^{\Ckplusone}(\ti^{C_{2^{n+1}}})\cdot a_{\lambda_k}^{2^{k-1}(2^i-1)}a_{\lambda_{k-1}}^{2^{k-2}(2^i-1)} \cdots a_{\lambda_2}^{2(2^i-1)} a_{\lambda_1}^{(2^i-1)}\right) \\
&= N_{C_2}^{C_{2^{k+j}}}(\ti^{C_{2^{n+1}}}) \cdot a_{\lambda_{k+j-1}}^{2^{k+j-2}(2^i-1)} a_{\lambda_{k+j-2}}^{2^{k+j-3}(2^i-1)} \cdots a_{\lambda_{j}}^{2^{j-1}(2^i-1)}\\
&= N_{C_2}^{C_{2^{k+j}}}(\ti^{C_{2^{n+1}}}) \cdot \left(\frac{a_{\rhobar_{2^{k+j}}}}{a_{\rhobar_{2^{k+j}/2^k}}} \right)^{2^i-1}.\qedhere
\end{align*}
\end{proof}

\section{The Transchromatic Isomorphism Theorem}
\label{sec:TranschromaticIsomTheorem}

In this section, we will prove the Transchromatic Isomorphism Theorem (\cref{thm:TranschromaticMain}), which establishes shearing isomorphisms between the slice spectral sequences of $\BPG \langle I \rangle$ for different groups $G$.  

Recall that if $V$ is an element in $RO(C_{2^{n+1}})$ and $1 \leq k \leq n+1$, then $V^{C_{2^k}} \in RO(C_{2^{n+1}}/C_{2^k})$ is the fixed points of $V$ with respect to the subgroup $C_{2^k} \subseteq \Cnplusone$.  Also, recall from \cref{df:LineLVh-1} that the line $\mathcal{L}^V_{2^k-1}$ on the $(V+t-s, s)$-graded page of an $RO(C_{2^{n+1}})$-graded spectral sequence is defined by the equation 
\[s = (2^k-1)(t-s) + \tau_V(\mathcal{F}_{\leq 2^k}),\]
where by \cref{df:HmaxHmintau}, 
\[\tau_V(\mathcal{F}_{\leq 2^k}) = \max_{0 \leq j \leq k} \left(|V^{C_{2^j}}| \cdot 2^j - |V| \right).\]

\begin{df} \label{df:ConstantCinTranschromaticTheorem}\rm
For $V \in RO(C_{2^{n+1}})$ and $1 \leq k \leq n$, let $C_{V,k}$ be the non-negative constant defined by the equation
\[C_{V,k} = \frac{\tau_V(\mathcal{F}_{\leq 2^k}) - \left(|V^{C_{2^k}}|\cdot 2^k-|V|\right)}{2^k} \geq 0. \]
\end{df}

\begin{thm}[Shearing Isomorphism]\label{thm:TranschromaticMain}
Suppose $I \subseteq \N$.  For all $V \in RO(C_{2^{n+1}})$ and $1 \leq k \leq n$, there is a shearing isomorphism $d_{2^kr - (2^k-1)} \leftrightsquigarrow d_r$ between the following regions of spectral sequences:
\begin{enumerate}
\item The $(V+t-s, s)$-graded page of ${\SliceSS(\BPCnplusone \langle I \rangle)}$ on or above the line $\mathcal{L}^V_{2^k-1}$ within the region $t-s \geq 0$; and 
\item The $(V^{C_{2^k}} + t-s, s)$-graded page of $\SliceSS(\BPCnminuskplusone \langle I \rangle)$ on or above the horizontal line $s = C_{V, k}$ within the region $t -s \geq 0$.  
\end{enumerate}
More precisely, let $\E_{\Cnplusone}$ and $\E_{\Cnminuskplusone}$ represent the slice spectral sequences of  $\BPCnplusone \langle I \rangle$ and $\BPCnminuskplusone \langle I \rangle$, respectively.  Then the following statements hold: 
\begin{enumerate}
\item In the specified regions of their respective $\mathcal{E}_2$-pages, there is an isomorphism 
\begin{equation}\label{eq:TranschromaticMainIsom}
\E_{\Cnplusone, 2}^{s', V+t'} \cong \Pb_{\Cnplusone/\Ck}\left(\E_{\Cnminuskplusone, 2}^{s, V^{\Ck}+t}\right),
\end{equation}
where
\begin{align*}
t' &= \left(|V^{\Ck}|\cdot 2^k-|V|\right) + 2^k t, \\
s' &= \left(|V^{\Ck}|\cdot 2^k-|V|\right) + (2^k-1)(t-s) +2^k s.
\end{align*}

\item In the specified region for $\E_{\Cnplusone}$, all the differentials have lengths of the form 
\[2^kr - (2^k-1), \,\,\, r \geq 2.\] 

\item The isomorphism (\ref{eq:TranschromaticMainIsom}) induces a one-to-one correspondence between the $d_{2^kr - (2^k-1)}$-differentials in $\E_{\Cnplusone}$ originating on or above the line $\mathcal{L}^V_{2^k-1}$ and the $d_r$-differentials in $\E_{\Cnminuskplusone}$ originating on or above the horizontal line $s = C_{V, k}$. 
\end{enumerate}
\end{thm}

In \cref{thm:TranschromaticMain}, for every $V \in RO(\Cnplusone)$, we have shearing isomorphisms between $RO(G)$-graded regions of spectral sequences.  Certain choices of $V$ will be particularly useful for our applications.  Before presenting the proof of \cref{thm:TranschromaticMain}, we will state some immediate corollaries that will be useful for our discussions in subsequent sections.

\begin{cor}\label{cor:TranschromaticVinRO(G/N)}
Suppose $I \subseteq \N$ and $V \in RO(\Cnplusone/\Ck)$ with $|V| = 0$.  We can view $V$ as an element in $RO(\Cnplusone)$ through pullback along the map $\Cnplusone \to \Cnplusone/\Ck$.  There is a shearing isomorphism ${d_{2^kr - (2^k-1)} \leftrightsquigarrow d_r}$ between the following regions of spectral sequences:
\begin{enumerate}
\item The $(V+t-s, s)$-graded page of $\SliceSS(\BPCnplusone \langle I \rangle)$ on or above the line $\mathcal{L}_{2^k-1}$ (the line of slope $(2^k-1)$ through the origin $(V, 0)$) within the region $t-s \geq 0$; and 
\item The $(V+t-s,s)$-graded page of $\SliceSS(\BPCnminuskplusone \langle I \rangle)$ on or above the line $s = 0$ within the region $t-s \geq 0$.  
\end{enumerate}
\end{cor}
\begin{proof}
As an element in $RO(\Cnplusone)$, $V$ is fixed by $C_j$ for all $0 \leq j \leq k$.  Therefore $|V^{C_j}| = |V| = 0$ for all $0 \leq j \leq k$.  This implies that the constants $\tau_V(\mathcal{F}_{\leq 2^k})$ and $C_{V, k}$ are both 0.  The claim now follows directly from \cref{thm:TranschromaticMain}.  
\end{proof}

\begin{cor}\label{cor:TranschromaticIntegerGraded}
Suppose $I \subseteq \N$.  On the integer-graded page, there is a shearing isomorphism ${d_{2^kr - (2^k-1)} \leftrightsquigarrow d_r}$ between the following regions of spectral sequences:
\begin{enumerate}
\item The slice spectral sequence of $\BPCnplusone \langle I \rangle$ on or above the line $\mathcal{L}_{2^k-1}$ (the line of slope $(2^k-1)$ through the origin); and 
\item The entire slice spectral sequence of $\BPCnminuskplusone \langle I \rangle$.  
\end{enumerate}
\end{cor}
\begin{proof}
By \cite[Theorem~C]{MeierShiZengStratification}, both slice spectral sequences are concentrated in the region $t-s \geq 0$.  The claim now follows directly from \cref{cor:TranschromaticVinRO(G/N)} by setting $V = 0$.  
\end{proof}

\begin{exam}\rm
Consider the slice spectral sequence of $\BPCfour\langle 1 \rangle$, as shown in \cref{fig:SliceSSBPCfourone}.  Letting $V = 0$ and $I = \{1\}$ in \cref{thm:TranschromaticMain} shows that there is a shearing isomorphism $d_{2r-1} \leftrightsquigarrow d_r$ between the integer-graded page of $\SliceSS(\BPCfour \langle 1 \rangle)$ on or above the line $\mathcal{L}_1$ and the integer-graded page of $\SliceSS(\BPR \langle 1 \rangle)$ (\cref{fig:SliceSSBPRone}) on or above the line $s = 0$.  This is the first instance of the shearing isomorphism that was observed by Hill--Hopkins--Ravenel that motivated the authors to embark on the current project. 
\begin{figure}
\begin{center}
\makebox[\textwidth]{\hspace{0in}\includegraphics[trim={0cm 12cm 5cm 4cm}, clip, scale = 0.8]{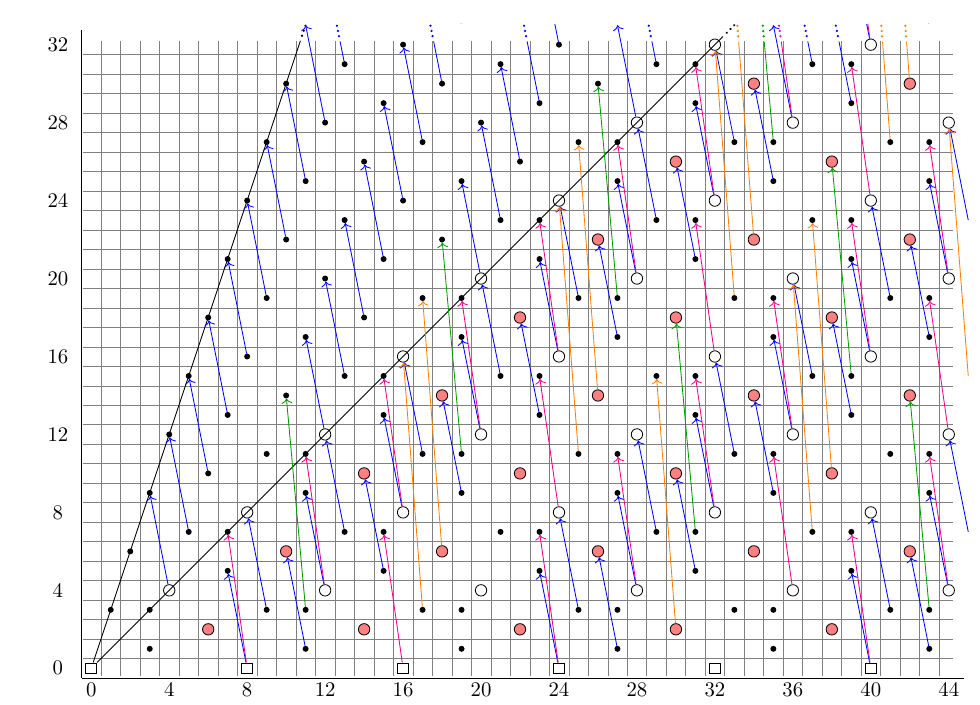}}
\caption{The slice spectral sequence of $\BPCfour \langle 1 \rangle$.}
\hfill
\label{fig:SliceSSBPCfourone}
\end{center}
\end{figure}

\begin{figure}
\begin{center}
\makebox[\textwidth]{\hspace{0in}\includegraphics[trim={3cm 15.5cm 5cm 4cm}, clip, scale = 0.7]{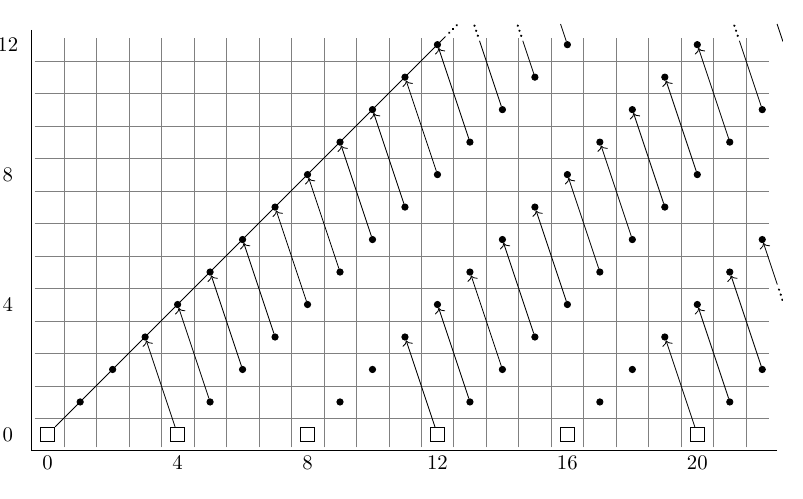}}
\caption{The slice spectral sequence of $\BPR \langle 1 \rangle$.}
\hfill
\label{fig:SliceSSBPRone}
\end{center}
\end{figure}
\end{exam}

\begin{cor}\label{cor:UVmapstoUVTranschromatic}
Suppose $I \subseteq \N$ and $V \in RO(\Cnplusone/\Ck)$ is of the form $V = c_0 (2\sigma) + c_1 \lambda_1 + \cdots + c_{n-k}\lambda_{n-k}$, where $c_i \geq 0$ for all $0 \leq i \leq n-k$.  Then the orientation class $u_V$ is a permanent cycle in the slice spectral sequence of $\BPCnplusone \langle I \rangle$ if and only if $u_V$ is a permanent cycle in the slice spectral sequence of $\BPCnminuskplusone \langle I \rangle$.  
\end{cor}
\begin{proof}
Since $|u_V| = |V|- V$, it is in position $(|V|-V, 0)$ on the $(|V|-V + t-s, s)$-graded page.  The claim now follows from applying \cref{cor:TranschromaticVinRO(G/N)} for the ${\left(|V| - V+t-s, s\right)}$-graded page and \cref{thm:CorrespondenceFormula} (2).  
\end{proof}

As a consequence of \cref{thm:TranschromaticMain} and its corollaries, we give an alternative proof of the Slice Differentials Theorem by Hill, Hopkins, and Ravenel \cite[Theorem~9.9]{HHR}.  In our proof, we utilize the $C_2$-slice differentials from the slice spectral sequence of $\BPR$ as an initial input.  By applying the shearing isomorphism, we then establish a family of differentials in the slice spectral sequence of $\BPCnplusone$ for all $n \geq 1$.  

\begin{thm}[Hill--Hopkins--Ravenel Slice Differentials Theorem]\label{thm:HHRSliceDiffTheorem} $\hfill$ \\
For all $n \geq 0$, the following family of differentials exist in the slice spectral sequence for $\BPCnplusone$:
\[d_{2^{n+1}(2^{i}-1)+1}(u_{2\sigma}^{2^{i-1}}) = N_{C_2}^{C_{2^{n+1}}}(\ti^{C_{2^{n+1}}}) a_{\rhobar_{2^{n+1}}}^{2^i-1} a_\sigma^{2^i}, \,\,\, i \geq 1. \]
\end{thm}
\begin{proof}
When $n = 0$, it is a consequence of the computation in \cite{HuKriz} that in the slice spectral sequence for $\BPR$, we have the differentials 
\[d_{2^{i+1}-1}(u_{2\sigma}^{2^{i-1}}) = \ti^{C_2} a_{\sigma}^{2^{i+1}-1}, \,\,\, i \geq 1.\]
Letting $I = \mathbb{N}$, $k = n$, and $V = |u_{2\sigma}^{2^{i-1}}|= (2^i-2^i\sigma) \in RO(\Cnplusone/\Cn)$ in \cref{cor:TranschromaticVinRO(G/N)} shows that there is a shearing isomorphism $d_{2^nr - (2^n-1)} \leftrightsquigarrow d_r$ between the following regions of spectral sequences:  
\begin{enumerate}
\item The $(2^i-2^i\sigma + t-s, s)$-graded page of $\SliceSS(\BPCnplusone)$ on or above the line $\mathcal{L}_{2^n-1}$ (the line of slope $(2^n-1)$ through the origin $(2^i-2^i\sigma, 0)$) within the region $t-s \geq 0$; and 
\item The $(2^i-2^i\sigma + t-s, s)$-graded page of $\SliceSS(\BPR)$ on or above the line $s = 0$ within the region $t-s \geq 0$.  
\end{enumerate}
Since the $C_2$-differential on $u_{2\sigma}^{2^{i-1}}$ is within this isomorphism region, it corresponds to a $\Cnplusone$-differential in $\SliceSS(\BPCnplusone)$ of length 
\[2^n\cdot (2^{i+1}-1) - (2^n-1) = 2^{n+1}(2^i-1)-1.\]
Using the correspondence formulas in \cref{thm:CorrespondenceFormula}, we have the following correspondences: 
\begin{enumerate}
\item $a_\sigma^{2^{i+1}-1} \leftrightsquigarrow a_\sigma^{2^{i+1}-1}$;
\item $u_{2\sigma}^{2^{i-1}} \leftrightsquigarrow u_{2\sigma}^{2^{i-1}}$;
\item $\ti^{C_2} \leftrightsquigarrow N_{C_2}^{\Cnplusone}(\ti^{\Cnplusone})\cdot \left(\frac{a_{\bar{\rho}_{2^{n+1}}}}{a_\sigma} \right)^{2^i-1}$.  
\end{enumerate}
Therefore, the shearing isomorphism produces the $\Cnplusone$-differentials 
\begin{align*}
d_{2^{n+1}(2^{i}-1)+1}(u_{2\sigma}^{2^{i-1}}) &= N_{C_2}^{\Cnplusone}(\ti^{\Cnplusone})\cdot \left(\frac{a_{\bar{\rho}_{2^{n+1}}}}{a_\sigma} \right)^{2^i-1} \cdot a_\sigma^{2^{i+1}-1} \\
&= N_{C_2}^{C_{2^{n+1}}}(\ti^{C_{2^{n+1}}}) a_{\rhobar_{2^{n+1}}}^{2^i-1} a_\sigma^{2^i}
\end{align*}
for all $i \geq 1$.
\end{proof}

\begin{proof}[Proof of \cref{thm:TranschromaticMain}] 
The following diagram outlines the steps of our proof: 
\begin{equation} \label{diag:TranschromaticProof}
\begin{tikzcd}
P_\bullet\left(\BPCnminuskplusone \langle I \rangle \right) \ar[ddddd, leftrightsquigarrow, "d_r \leftrightsquigarrow d_{2^kr - (2^k-1)}"', "\text{ shearing isom.}"] \ar[rrr, "\varphi_1", "\text{Corollary}~\ref{cor:sliceRecoveryExample1}\,\,(s=0)"'] &&& \EF[C_2] \wedge P_\bullet\left(\BPCnminuskplusone \langle I \rangle \right) \ar[ddd, "\text{ shearing isom.}", leftrightsquigarrow, "\text{Theorem}~\ref{thm:ShearingGeneralTheory}"'near start, "d_r \leftrightsquigarrow d_{2^kr - (2^k-1)}"'] \\ \\ \\
&&& \Pb_{C_{2^{n+1}}/C_{2^k}} \D^k \left(\EF[C_2] \wedge P_\bullet(\BPCnminuskplusone \langle I \rangle) \right) \ar[dd, dash, "\, \simeq", "\text{Theorem}~\ref{thm:DualTowerEquivalenceMain}"'] \\ \\
P_\bullet\left(\BPCnplusone \langle I \rangle\right) \arrow[rrr, "\varphi_{2^{k}}", "\text{Theorem}~\ref{thm:SliceRecovery1}\,\,\left(\mathcal{L}^V_{2^k-1}\right)"'] &&& \EF[C_{2^{k+1}}] \wedge P_\bullet(\BPCnplusone \langle I \rangle). 
\end{tikzcd}
\end{equation}
By \cref{cor:sliceRecoveryExample1}, the map $\varphi_1$ induces an isomorphism of spectral sequences on or above the horizontal line $s = 0$ on the $(V^{C_{2^k}}+t-s,s)$-graded page.  \cref{thm:ShearingGeneralTheory} then establishes a shearing isomorphism $d_r \leftrightsquigarrow d_{2^kr - (2^k-1)}$ between the spectral sequence $\mathcal{E}_{P_\bullet}$ that is associated to the tower 
\[\EF[C_2] \wedge P_\bullet\left(\BPCnminuskplusone \langle I \rangle \right)\]
and the spectral sequence $\mathcal{E}_{Q_\bullet}$ that is associated to the tower 
\[\Pb_{C_{2^{n+1}}/C_{2^k}} \D^k \left(\EF[C_2] \wedge P_\bullet(\BPCnminuskplusone \langle I \rangle) \right),\]
which, by \cref{thm:DualTowerEquivalenceMain}, is also the spectral sequence that is associated to the tower 
\[\EF[C_{2^{k+1}}] \wedge P_\bullet(\BPCnplusone \langle I \rangle).\]
More explicitly, there is an isomorphism 
\[\E_{Q_\bullet, 2}^{s', V+t'} \cong \Pb_{\Cnplusone/\Ck}\left(\E_{P_\bullet, 2}^{s, V^{\Ck}+t}\right) \,\,\,\,\, (r \geq 2), \]
where
\begin{align*}
t' &= \left(|V^{\Ck}|\cdot 2^k-|V|\right) + 2^k t, \\
s' &= \left(|V^{\Ck}|\cdot 2^k-|V|\right) + (2^k-1)(t-s) +2^k s.
\end{align*}
This shearing isomorphism induces a one-to-one correspondence between the $d_r$-differentials in $\mathcal{E}_{P_\bullet}$ and the $d_{2^kr - (2^k-1)}$-differentials in $\mathcal{E}_{Q_\bullet}$.  Moreover, under this shearing isomorphism, the horizontal line $s = C_{V, k}$ is sheared to the line 
\begin{align*}
s' &= \left(|V^{\Ck}|\cdot 2^k-|V|\right) + (2^k-1)(t'-s') + 2^k \cdot C_{V,k}  \\
&= (2^k-1)(t'-s') + \tau_V(\mathcal{F}_{\leq 2^k}),
\end{align*}
which is exactly $\mathcal{L}^V_{2^k-1}$.  

By the Slice Recovery Theorem (\cref{thm:SliceRecovery1}), the map $\varphi_{2^k}$ induces an isomorphism of spectral sequences on or above the line $\mathcal{L}^V_{2^k-1}$ in the region $t'-s' \geq 0$.  This finishes the proof of the theorem.
\end{proof}

\begin{rmk}\rm
In the proof of \cref{thm:TranschromaticMain}, we observe that (\ref{eq:TranschromaticMainIsom}) induces an isomorphism for classes that are above the lines $\mathcal{L}_{2^k-1}^V$ and $s = C_{V,k}$, and there is some fringe behavior for classes that are on these two lines.  More specifically, for classes on the $\mathcal{E}_2$-pages that lie on the lines $\mathcal{L}_{2^k-1}^V$ and $s = C_{V,k}$, (\ref{eq:TranschromaticMainIsom}) induces an isomorphism between them if they map to nontrivial classes in the corresponding localized slice spectral sequences under the maps 
\[\SliceSS(\BPCnplusone \langle I \rangle) \longrightarrow \EF[\Ckplusone] \wedge \SliceSS(\BPCnplusone \langle I \rangle)\]
and 
\[\SliceSS(\BPCnminuskplusone \langle I \rangle) \longrightarrow \EF[C_2] \wedge \SliceSS(\BPCnminuskplusone \langle I \rangle).\]
According to \cite[Corollary 2.3]{MeierShiZengStratification}, these maps induce surjections on the $\mathcal{E}_2$-pages for classes lying on these two lines.  Furthermore, by \cref{thm:SliceRecovery1}, the classes in the kernel of the aforementioned maps do not support nontrivial differentials.  Therefore, on the lines $\mathcal{L}_{2^k-1}^V$ and $s = C_{V,k}$, (\ref{eq:TranschromaticMainIsom}) induces an isomorphism between classes that could potentially support nontrivial differentials.
\end{rmk}

\section{The transchromatic tower} \label{sec:TranschromaticTowerLubinTateTheories}

In this section, we will use \cref{thm:TranschromaticMain} to study transchromatic phenomena in the slice spectral sequences of $\BPCnplusone \langle m \rangle$ and $\BPCnplusone \langle m, m \rangle$.  Recall from the introduction that Beaudry, Hill, the second author, and Zeng \cite{BHSZ} used the $\Cnplusone$-equivariant formal group law associated with $\BPCnplusone \langle m \rangle$ to construct explicit $\Cnplusone$-models of $E_{2^n \cdot m}$.  From their construction, the equivariant orientation of \cite{HahnShi} factors through the quotient $\BPCnplusone \langle m \rangle$: 
\[\begin{tikzcd}
\BPCnplusone \ar[r] \ar[d] & E_h \\ 
\BPCnplusone \langle m \rangle \ar[ru]
\end{tikzcd}\]
In light of this construction, the spectrum $\BPCnplusone \langle m \rangle$ can be regarded as a connective model for the $\Cnplusone$-spectrum $E_h$.  

\begin{construction}[Transchromatic tower] \rm \label{construction:TranschromaticTower}
In \cref{construction:Tower}, the stratification tower~(\ref{diagram:SliceSSTower}) for $\BPCnplusone \langle m \rangle$ is 
\[\begin{tikzcd}
\SliceSS\left(\BPCnplusone \langle m \rangle\right) \ar[rr, "\mathcal{L}_0^V"] \ar[rrd, swap, "\mathcal{L}_{1}^V"] \ar[rrddd, swap, bend right = 7, "\mathcal{L}_{2^{n-1}-1}^V"] \ar[rrdddd, bend right = 30, swap, "\mathcal{L}_{2^n-1}^V"]&& \EF[C_2] \wedge \SliceSS\left(\BPCnplusone \langle m \rangle\right) \ar[d, "\mathcal{L}_{1}^V"]  \\
&& \EF[C_4] \wedge \SliceSS\left(\BPCnplusone \langle m \rangle\right) \ar[d, "\mathcal{L}_{3}^V"]  \\
&& \vdots \ar[d, "\mathcal{L}_{2^{n-1}-1}^V"] \\ 
&& \EF[C_{2^n}] \wedge \SliceSS\left(\BPCnplusone \langle m \rangle\right) \ar[d, "\mathcal{L}_{2^{n}-1}^V"] \\
&& \EF[C_{2^{n+1}}] \wedge \SliceSS\left(\BPCnplusone \langle m \rangle\right). 
\end{tikzcd}\]
For each localized slice spectral sequence 
\[\EF[\Ckplusone] \wedge \SliceSS\left(\BPCnplusone \langle m \rangle \right), \, 1 \leq k \leq n\]
in the tower above, the Transchromatic Isomorphism Theorem (\cref{thm:TranschromaticMain}) uses the shearing isomorphism 
\[\EF[\Ckplusone] \wedge \SliceSS\left(\BPCnplusone \langle m \rangle \right) \leftrightsquigarrow \EF[C_2] \wedge \SliceSS\left(\BPCnminuskplusone \langle m \rangle \right)\]
from \cref{thm:DualTowerEquivalenceMain} to establish shearing isomorphisms between the corresponding slice spectral sequences:
\[\begin{tikzcd}
\SliceSS\left(\BPCnplusone \langle m \rangle \right) 
\ar[rr, leftrightsquigarrow, "\mathcal{L}_{2^k-1}^V"] && \SliceSS\left(\BPCnminuskplusone \langle m \rangle \right) 
\end{tikzcd}\]
Here, the label $\mathcal{L}_{2^k-1}^V$ indicates the shearing isomorphism region in the slice spectral sequence of $\BPCnplusone \langle m \rangle$.  The shearing isomorphisms $\mathcal{L}_{2^k-1}^V$ form the following transchromatic tower: 
\begin{equation}\label{diagram:TranschromaticTower}
\begin{tikzcd}
\SliceSS\left( \BPCnplusone \langle m \rangle \right) \,\,\, \left(E_h^{h\Cnplusone}\right) \ar[r, leftrightarrow, "\mathcal{L}_1^V"] \ar[rd, leftrightarrow, swap, "\mathcal{L}_3^V"] \ar[rddd, leftrightarrow, swap, bend right = 7, "\mathcal{L}_{2^{n-1}-1}^V"] \ar[rdddd, leftrightarrow, swap, bend right = 30, "\mathcal{L}_{2^{n}-1}^V"]& \SliceSS\left(\BPCn \langle m \rangle \right) \,\,\, \left(E_{h/2}^{h\Cn}\right)\ar[d, leftrightarrow] \\ 
& \SliceSS\left(\BPCnminusone \langle m \rangle \right) \,\,\, \left(E_{h/4}^{hC_{2^{n-1}}}\right)\ar[d, leftrightarrow] \\ 
& \vdots \ar[d, leftrightarrow] \\ 
& \SliceSS\left( \BPCfour \langle m \rangle \right) \,\,\, \left(E_{h/2^{n-1}}^{hC_4}\right)\ar[d, leftrightarrow] \\ 
& \SliceSS\left( \BPR \langle m \rangle \right) \,\,\, \left(E_{h/2^{n}}^{hC_2}\right).
\end{tikzcd}
\end{equation}
\end{construction}

The transchromatic tower (\ref{diagram:TranschromaticTower}) provides an inductive method for computing the slice spectral sequences of the higher height theories from the lower height theories (see \cref{fig:PicTranschromaticTower}).  More specifically, when computing the slice spectral sequence of $\BPCnplusone \langle m \rangle$, which serves as the $\Cnplusone$-equivariant connective model for $E_h$, the tower reveals that the region on or above the line $\mathcal{L}_{2^k-1}^V$, where $1 \leq k \leq n$, is completely determined by the slice spectral sequence of $\BPCnminuskplusone \langle m \rangle$, the $\Cnminuskplusone$-equivariant connective model for $E_{h/2^k}$.  

\begin{figure}
\begin{center}
\makebox[\textwidth]{\includegraphics[trim={0cm 0cm 0cm 0cm}, clip, scale = 0.6]{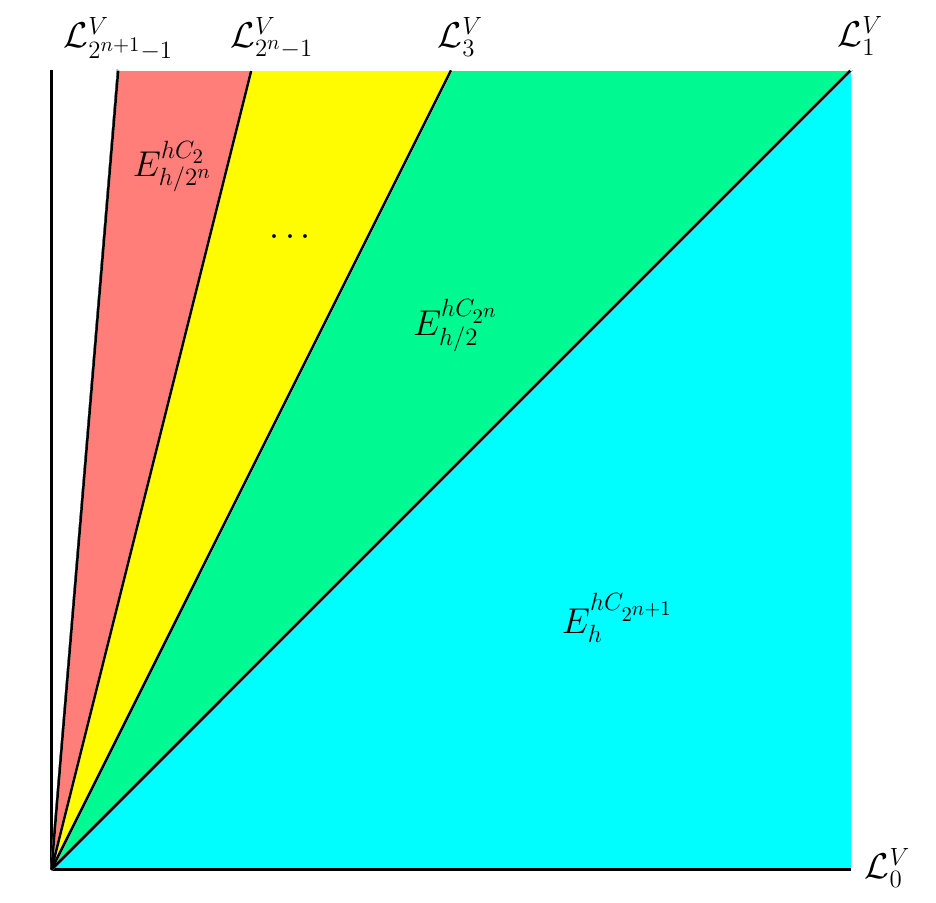}}
\caption{Regions of the slice spectral sequence of $E_h$ corresponding to the transchromatic tower.}
\hfill
\label{fig:PicTranschromaticTower}
\end{center}
\end{figure}

\begin{exam}\rm
For all $m \geq 1$, $\BPR \langle m \rangle$ is the $C_2$-equivariant connective model for $E_m$, and $\BPCnplusone \langle m \rangle$ is the $\Cnplusone$-connective model for $E_{2^n \cdot m}$.  The computation of the $C_2$-slice spectral sequence of $\BPR \langle m \rangle$ by Hu and Kriz \cite{HuKriz} completely determines the $RO(G)$-graded page of the slice spectral sequence of $\BPCnplusone \langle m \rangle$ on or above the line $\mathcal{L}^V_{2^n-1}$.
\end{exam}

\begin{exam}\rm
For $n \geq 0$, the $\Cnplusone$-spectrum $\BPCnplusone \langle 1 \rangle$ serves as the $\Cnplusone$-equivariant connective model for $E_{2^n}$.  The computation of the slice spectral sequence of $\BPCfour \langle 1 \rangle$ ($E_2^{hC_4}$) by Hill, Hopkins, and Ravenel \cite{HHRKH} completely determines the integer-graded spectral sequence of $\BPCnplusone \langle 1 \rangle$ ($E_{2^n}^{h\Cnplusone}$) on or above the line of slope $(2^{n-1}-1)$ through the origin.

In particular, when $n = 2$, in the slice spectral sequence of $\BPCeight \langle 1 \rangle$ ($E_4^{hC_8}$), which detects all the Kervaire invariant elements, the differentials on or above the line of slope 1 are completely determined by the differentials in the slice spectral sequence of $\BPCfour \langle 1 \rangle$.  
\end{exam}

\begin{exam} \rm
For $n \geq 0$, the $\Cnplusone$-spectrum $\BPCnplusone \langle 2 \rangle$ serves as the $\Cnplusone$-equivariant connective model for $E_{2^{n+1}}$.  The computation of the slice spectral sequence of $\BPCfour \langle 2 \rangle$ ($E_4^{hC_4}$) by Hill, the second author, Wang, and Xu \cite{HSWX} completely determines the integer-graded spectral sequence of $\BPCnplusone \langle 2 \rangle$ ($E_{2^{n+1}}^{h\Cnplusone}$) on or above the line of slope $(2^{n-1}-1)$ through the origin.
\end{exam}

\begin{rmk}\rm
In \cite{BHLSZ}, Beaudry, Hill, Lawson, the second author, and the third author constructed and studied the $\Cnplusone$-equivariant theories $\BPCnplusone \langle m, m \rangle$.  These theories are the equivariant analogues of the integral connective Morava $K$-theories.  One notable feature of these theories is that while the non-equivariant connective Morava $K$-theory $K(m) = BP \langle m, m \rangle$ has chromatic heights concentrated at levels $0$ and $m$, the theories $\BPCnplusone \langle m, m \rangle$ have chromatic heights concentrated at levels $0$, $m$, $2m$, $\ldots$, $2^n \cdot m$ \cite[Theorem~3.1]{BHLSZ}. 

\cref{thm:TranschromaticMain} establishes a transchromatic tower for the theories $\BPCnplusone \langle m, m \rangle$.  In particular, everything on or above the line $\mathcal{L}_1^V$ in the slice spectral sequence for $\BPCnplusone \langle m, m \rangle$, which captures chromatic information at heights $0$, $m$, $\ldots$, $2^n\cdot m$, is completely determined by the slice spectral sequence for $\BPCn \langle m, m \rangle$, which captures chromatic information at heights $0$, $m$, $\ldots$, $2^{n-1}\cdot m$.  
\end{rmk}

\section{\texorpdfstring{$RO(G)$}{text}-graded periodicities} \label{sec:ROGPeriodicity}

In this section, we will apply the Transchromatic Isomorphism Theorem (\cref{thm:TranschromaticMain}) to establish $RO(G)$-graded periodicities of $E_h$ from $RO(G/{C_2})$-graded periodicities of $E_{h/2}$.  

For $G = \Cn$, Beaudry--Hill--Shi--Zeng \cite{BHSZ} used the universal $\Cn$-equivariant formal group law of $\BPCn$ to construct $\Cn$-equivariant models of the height-$h$ Lubin--Tate theory $E_h$, where $h = 2^{n-1}\cdot m$.  For our discussion in this section, we will focus on the following two Lubin--Tate theories: 
\begin{enumerate}
\item The $\Cn$-equivariant theory $E_{2^{n-1} \cdot m}$; and 
\item The $\Cnplusone$-equivariant theory $E_{2^n \cdot m}$.  
\end{enumerate}
By work of Hahn and the second author \cite{HahnShi}, these models are equipped with equivariant orientations from $\BPCn$ and $\BPCnplusone$, respectively. It is a consequence of the construction in \cite{BHSZ} that these equivariant orientations factor through certain localizations and quotients of $\BPCn$ and $\BPCnplusone$, which we will now discuss.


For $1 \leq k \leq n$, the $\ti^{\Ck}$-generators are originally defined to be elements in $\pi_{*\rho_2}^{C_2} i_{C_2}^*\BPCk$.  The unit map 
\[\BPCk \longrightarrow i_{\Ck}^*N_{\Ck}^{\Cn}(\BPCk) = i_{\Ck}^*\BPCn\]
of the norm-restriction adjunction enables us to treat the $\ti^{\Ck}$-generators as elements in $\pi_{*\rho_2}^{C_2} i_{C_2}^* \BPCn$ through the composition 
\[\ti^{\Ck}: S^{(2^i-1)\rho_2} \longrightarrow i_{C_2}^*\BPCk \longrightarrow i_{C_2}^* \BPCn.\]
Similarly, we can treat the $\ti^{\Ckplusone}$-generators as elements in $\pi_{*\rho_2}^{C_2} i_{C_2}^*\BPCnplusone$ through the composition 
\[\ti^{\Ckplusone}: S^{(2^i-1)\rho_2} \longrightarrow i_{C_2}^*\BPCkplusone \longrightarrow i_{C_2}^* \BPCnplusone.\]
We note that under this more general treatment, the equality 
\[\Phi^{C_2}(\ti^{\Ck}) = \Phi^{C_4} N_{C_2}^{C_4} (\ti^{\Ckplusone})\]
proven in \cref{lem:tiGeneratorRelations} and the correspondence formulas proven in \cref{prop:tiCorrespondenceFormula} and \cref{thm:CorrespondenceFormula} still hold.  

\begin{df}\label{df:DCnDCnplusone}\rm
Define $D_{\Cn, m} \in \pi_{*\rho_{2^n}}^{\Cn} \BPCn$ to be the element
\begin{align*}
D_{\Cn, m} = \prod_{k =1}^n N_{C_2}^{\Cn}(\tee^{\Ck}_{2^{n-k}\cdot m}) = N_{C_2}^{\Cn}(\tee^{C_2}_{2^{n-1} \cdot m}) \cdots N_{C_2}^{\Cn}(\tee^{\Cn}_{m}).
\end{align*}
Furthermore, define $\overline{D}_{\Cn, m} \in \pi_{*\rho_{2^{n}}}^{\Cn} \BPCn$ to be the element 
\[\overline{D}_{\Cn, m} = \prod_{k = 2}^{n} N_{C_2}^{\Cn}(\tee^{\Ck}_{2^{n-k}\cdot m}) = N_{C_2}^{\Cn}(\tee^{C_4}_{2^{n-2}\cdot m}) \cdots N_{C_2}^{\Cn}(\tee^{\Cn}_{m}).\]
Note that with our definition, $D_{\Cn, m} = N_{C_2}^{\Cn}(\tee^{C_2}_{2^{n-1}\cdot m}) \cdot \overline{D}_{\Cn, m}$. 
\end{df}


It is a consequence from the constructions of $E_{2^{n-1}\cdot m}$ and $E_{2^n \cdot m}$ in \cite{BHSZ} that the equivariant orientations $\BPCn \to E_{2^{n-1}\cdot m}$ and $\BPCnplusone \to E_{2^n \cdot m}$ factor through the localizations $D_{\Cn, m}^{-1} \BPCn$ and $D_{\Cnplusone, m}^{-1} \BPCnplusone$, and further through the quotients $D_{\Cn, m}^{-1} \BPCn \langle m \rangle$ and $D_{\Cnplusone, m}^{-1} \BPCnplusone \langle m \rangle$.  There are diagrams
\begin{equation*}
\begin{tikzcd}
D_{\Cn, m}^{-1} \BPCn \ar[r] \ar[d] & E_{2^{n-1} \cdot m} \\ 
D_{\Cn, m}^{-1} \BPCn \langle m \rangle \ar[ru]
\end{tikzcd}
\end{equation*}
and 
\begin{equation*}
\begin{tikzcd} 
D_{\Cnplusone, m}^{-1} \BPCnplusone \ar[r] \ar[d] & E_{2^n \cdot m} \\
D_{\Cnplusone, m}^{-1} \BPCnplusone \langle m \rangle \ar[ru] 
\end{tikzcd}
\end{equation*}

\begin{lem}\label{lem:DinverseLocalizedSSEquiv}
We have the following equivalence of towers:     
\begin{multline*}
\EF[C_4] \wedge P_\bullet\left(\overline{D}_{\Cnplusone, m}^{-1} \BPCnplusone \langle m \rangle \right) \\
\simeq  \Pb_{C_{2^{n+1}}/C_2}\D \left( \EF[C_2] \wedge P_\bullet (D_{\Cn, m}^{-1}\BPCn \langle m \rangle) \right).    
\end{multline*}
\end{lem}
\begin{proof}
By setting $I = \{1, \ldots, m\}$ and $k = 1$ in \cref{thm:DualTowerEquivalenceMain}, we obtain the equivalence 
\begin{equation}\label{eq:DinverseLocSSEquiv}
\EF[C_{4}] \wedge P_\bullet\left(\BPCnplusone \langle m \rangle \right) \simeq \Pb_{C_{2^{n+1}}/C_2}\D \left( \EF[C_2] \wedge P_\bullet (\BPCn \langle m \rangle) \right).
\end{equation}
Under this equivalence, \cref{prop:tiCorrespondenceFormula} establishes the correspondence
\[N_{C_2}^{\Ckplusone}(\tee^{\Ckplusone}_{2^{n-k}\cdot m})a_{\lambda_k}^{2^{k-1}(2^{2^{n-k}\cdot m}-1)} \leftrightsquigarrow N_{C_2}^{\Ck}(\tee^{\Ck}_{2^{n-k}\cdot m})\]
for all $1 \leq k \leq n$.  Applying the norm, these correspondence formulas become
\[N_{C_2}^{\Cnplusone}(\tee^{\Ckplusone}_{2^{n-k}\cdot m})a_{\lambda_n}^{2^{n-1}(2^{2^{n-k}\cdot m}-1)} \leftrightsquigarrow N_{C_2}^{\Cn}(\tee^{\Ck}_{2^{n-k}\cdot m})\]
for all $1 \leq k \leq n$.  Multiplying these formulas together yields the correspondence 
\[\overline{D}_{\Cnplusone, m} \cdot a_{\lambda_n}^{\ell} \leftrightsquigarrow D_{\Cn, m},\]
where 
\[\ell = \sum_{1 \leq k \leq n} 2^{n-1}(2^{2^{n-k}\cdot m}-1).\]
Since smashing with $\EF[C_4]$ inverts the class $a_{\lambda_{n-1}}$, the class $a_{\lambda_n}$ is also inverted.  Therefore the term $a_{\lambda_n}^{\ell}$ is a unit in ${\EF[C_{4}] \wedge P_\bullet\left(\BPCnplusone \langle m \rangle \right)}$.  The desired equivalence now follows from Equivalence~(\ref{eq:DinverseLocSSEquiv}) and the correspondence formula above.  
\end{proof}

For the remainder of this section, we will use the notation $G=\Cnplusone$ and $h=2^n \cdot m$.  Suppose $V$ is an element in $RO(G)$ such that the orientation class $u_V$ in the slice spectral sequence of $D_{G, m}^{-1} \BPG \langle m \rangle$ is a permanent cycle.  Then under the map of spectral sequences
\[\SliceSS(D_{G, m}^{-1} \BPG \langle m \rangle) \longrightarrow \SliceSS(E_h),\]
its image $u_V \in \SliceSS(E_h)$ is also a permanent cycle.  Let $u \in \pi_{|V|-V}^G E_h$ denote any element in homotopy that is represented by $u_V$.  The composite
\[S^{|V|- V} \wedge E_h \xrightarrow{u \wedge \id} E_h \wedge E_h \xrightarrow{m} E_h\]
is a $G$-equivariant map between cofree (Borel) spectra.  Since the underlying non-equivariant map induces an equivalence $i_e^* E_h \simeq i_e^* E_h$, the composition map above induces a $G$-equivariant equivalence.  This produces a $(|V|-V)$-periodicity for $E_h$.

\begin{thm}\label{thm:ROGPeriodicityTranschromatic}
Suppose $V \in RO(G/C_2)$.  The class $u_V$ is a permanent cycle in the slice spectral sequence of $D_{G, m}^{-1} \BPG \langle m \rangle$ (and hence in the $G$-slice spectral sequence of $E_{h}$) if and only if the class $u_V$ is a permanent cycle in the slice spectral sequence of $D_{G/C_2, m}^{-1} \BPGmodtwo \langle m \rangle$ (and therefore in the $G/C_2$-slice spectral sequence of $E_{h/2}$).  Here, $V$ is treated as an element in $RO(G)$ through pullback along the map $G \to G/C_2$.
\end{thm}

\begin{rmk} \label{rmk:ROGPeriodicityBPGm}\rm
If the class $u_V$ is a permanent cycle in the slice spectral sequence of $D_{G, m}^{-1}\BPG$ or $\BPG\langle m \rangle$, then $u_V$ will also be a permanent cycle in the slice spectral sequence of $D_{G, m}^{-1} \BPG \langle m \rangle$ by naturality.  Therefore, the statement of \cref{thm:ROGPeriodicityTranschromatic} remains valid if we replace $D_{G, m}^{-1} \BPG \langle m \rangle$ by $D_{G, m}^{-1}\BPG$ or $\BPG\langle m \rangle$.
\end{rmk}

\begin{rmk}\rm
In light of \cref{rmk:ROGPeriodicityBPGm}, we can apply \cref{thm:ROGPeriodicityTranschromatic} iteratively to establish that if $V \in RO(G/\Ck)$, where $1 \leq k \leq n$, and $u_V$ is a permanent cycle in the slice spectral sequence of $\BPGmodCk \langle m \rangle$ (and therefore in the $G/\Ck$-slice spectral sequence of $E_{h/2^k}$), then $u_V$ remains a permanent cycle in the slice spectral sequence of $\BPG \langle m \rangle$ (and hence in the $G$-slice spectral sequence of $E_h$).  
\end{rmk}

\begin{proof}[Proof of \cref{thm:ROGPeriodicityTranschromatic}]
The equivalence of localized dual slice towers in \cref{lem:DinverseLocalizedSSEquiv} establishes a shearing isomorphism between the corresponding localized slice spectral sequences.  Consequently, this shearing isomorphism extends to the regions of the respective slice spectral sequences: 
\[\begin{tikzcd}
\SliceSS\left(\overline{D}_{G, m}^{-1} \BPG \langle m \rangle\right) \ar[r, leftrightsquigarrow, "\text{shearing}"]  \ar[d] & \SliceSS\left(D_{G/C_2, m}^{-1} \BPGmodtwo \langle m \rangle\right) \ar[d] \\ 
\EF[C_4] \wedge \SliceSS\left(\overline{D}_{G, m}^{-1} \BPG \langle m \rangle\right) \ar[r, leftrightsquigarrow, "\text{shearing}"] & \EF[C_2] \wedge \SliceSS\left(D_{G/C_2, m}^{-1} \BPGmodtwo \langle m \rangle\right).
\end{tikzcd}\]
The proof of \cref{cor:UVmapstoUVTranschromatic} applies to this context, showing that $u_V$ is a permanent cycle in the slice spectral sequence of $D_{G/C_2, m}^{-1} \BPGmodtwo \langle m \rangle$ if and only if $u_V$ is a permanent cycle in the slice spectral sequence of $D_{G, m}^{-1} \BPG \langle m \rangle$.
\end{proof}

\begin{rmk}\rm
Recall that the $RO(G)$-graded homotopy groups are graded by the representations $\{1, \sigma, \lambda_1, \ldots, \lambda_n\}$, while $RO(G/C_2)$-graded homotopy groups are graded by the representations $\{1, \sigma, \lambda_1, \ldots, \lambda_{n-1}\}$.  By considering the $RO(G/C_2)$-graded periodicities of the $G/C_2$-spectrum $E_{h/2}$ and applying \cref{thm:ROGPeriodicityTranschromatic}, we can immediately establish all the $RO(G)$-graded periodicities of the $G$-spectrum $E_h$ that are of the form $|V| - V$, where $V$ is generated by the representations $\{1, \sigma, \lambda_1, \ldots, \lambda_{n-1}\}$.
\end{rmk}

\begin{exam}\rm
By explicit computation of Hu and Kriz \cite{HuKriz}, it is shown that the class $u_{2^{m+1}\sigma}$ is a permanent cycle in the slice spectral sequence of $\BPR\langle m \rangle$.  This establishes a ${(2^{m+1} - 2^{m+1}\sigma)}$-graded periodicity for $E_m$ when considered as a $C_2$-spectrum.  Applying \cref{thm:ROGPeriodicityTranschromatic}, the class $u_{2^{m+1}\sigma}$ is also a permanent cycle in the slice spectral sequence of $\BPCnplusone \langle m \rangle$ for all $n \geq 0$.  As a consequence, a ${(2^{m+1} - 2^{m+1}\sigma)}$-graded periodicity is established for the $\Cnplusone$-spectrum $E_{2^n \cdot m}$ for all $n \geq 0$. 
\end{exam}

\begin{exam}\rm
The computation by Hill, Hopkins, and Ravenel~\cite{HHRKH} shows that the classes $u_{4\sigma}$, $u_{8\lambda_1}$, and $u_{4\lambda_1 2\sigma}$ are permanent cycles in the slice spectral sequence of $\BPCfour \langle 1 \rangle$.  These permanent cycles establish $RO(C_4)$-graded periodicities ${4 - 4 \sigma}$, ${16-8\lambda_1}$, and ${10-4\lambda_1 - 2 \sigma}$ for $E_2$ as a $C_4$-spectrum.  By applying \cref{thm:ROGPeriodicityTranschromatic}, we conclude that these classes are also permanent cycles in the slice spectral sequence of $\BPCnplusone \langle 1 \rangle$ for all $n \geq 1$.  Consequently, the $RO(\Cnplusone)$-graded periodicities 
\[4 - 4 \sigma, \, 16-8\lambda_1, \, 10-4\lambda_1 - 2 \sigma\]
are established for the $\Cnplusone$-spectrum $E_{2^n}$, where $n \geq 1$.
\end{exam}

\begin{exam}\rm
Similar to the previous example, Hill, the second author, Wang, and Xu~\cite{HSWX} computed the slice spectral sequence of $\BPCfour \langle 2 \rangle$ and showed that the classes $u_{8\sigma}$, $u_{32\lambda_1}$, and $u_{16\lambda_1 4 \sigma}$ are permanent cycles.  By applying \cref{thm:ROGPeriodicityTranschromatic}, we deduce that these classes are also permanent cycles in the slice spectral sequence of $\BPCnplusone \langle 2 \rangle$ for all $n \geq 1$.  Consequently, this establishes the $RO(\Cnplusone)$-graded periodicities 
\[8 - 8 \sigma, \, 64-32\lambda_1, \, 36-16\lambda_1 - 4 \sigma\]
for the $\Cnplusone$-spectrum $E_{2^{n}\cdot 2}$, where $n \geq 1$.
\end{exam}\rm

\section{Strong vanishing lines}\label{sec:VanishingLines}

In this section, we will maintain the notation $G = \Cnplusone$ and $h = 2^n \cdot m$.  In \cite{DuanLiShiVanishing}, Duan, Li, and the second author established a strong horizontal vanishing line of filtration $(2^{h+n+1} - 2^{n+1}+1)$ in the homotopy fixed point spectral sequence of $E_{h}$.  A key step in their proof relies on identifying isomorphism regions between the homotopy fixed point spectral sequence and the Tate spectral sequence \cite[Theorem~3.6]{DuanLiShiVanishing}.  

We can apply the exact same proof technique by substituting the homotopy fixed point spectral sequence with the equivariant slice spectral sequence, and the Tate spectral sequence with the localized slice spectral sequence $\widetilde{E}G \wedge \SliceSS$, while utilizing the slice recovery theorem (\cref{thm:SliceRecovery1}).  Consequently, this also establishes a strong horizontal vanishing line of filtration $(2^{h+n+1} - 2^{n+1}+1)$ in the slice spectral sequence of $E_h$.

We will use the Transchromatic Isomorphism Theorem (\cref{thm:TranschromaticMain}) to further establish strong vanishing lines of slope $(2^k-1)$ for all $0 \leq k \leq n$ in the slice spectral sequence of $E_h$.  

To start our discussion, recall that for $0 \leq k \leq n$, the $G/\Ck$-equivariant model of $E_{h/2^k}$ constructed in \cite{BHSZ} has a $G/\Ck$-equivariant orientation from $\BPGmodCk$ that factors through the localization $N_{C_2}^{G/\Ck}(\tee^{C_2}_{h/2^k})^{-1}\BPGmodCk$: 
\[\begin{tikzcd}
\BPGmodCk \ar[r] \ar[d] & E_{h/2^k}  \\
N_{C_2}^{G/\Ck}(\tee^{C_2}_{h/2^k})^{-1}\BPGmodCk \ar[ru]
\end{tikzcd}\]

\begin{lem}\label{lem:VanshingLinesLocalizedTowerEquiv}
For $1 \leq k \leq n$, we have the following equivalence of towers: 
\begin{multline*}
\EF[\Ckplusone] \wedge P_\bullet\left(N_{C_2}^{G}(\tee^{\Ckplusone}_{h/2^k})^{-1} \BPG \right) \\
\simeq \Pb_{G/\Ck} \D^k \left(\EF[C_2] \wedge P_\bullet\left(N_{C_2}^{G/\Ck}(\tee^{C_2}_{h/2^k})^{-1} \BPGmodCk\right) \right). 
\end{multline*}    
\end{lem}
\begin{proof}
The proof is almost exactly the same as that of \cref{lem:DinverseLocalizedSSEquiv}.  Consider the equivalence 
\begin{equation}\label{eq:VanishingLineTowerEquivalence}
\EF[\Ckplusone] \wedge P_\bullet \left(\BPG\right) \simeq \Pb_{G/\Ck} \D^k \left(\EF[C_2] \wedge P_\bullet\left(\BPGmodCk\right) \right)
\end{equation}
established in \cref{thm:DualTowerEquivalenceMain}.  Under this equivalence, letting $n = k$, $j=1$, and $i = h/2^k$ in \cref{thm:CorrespondenceFormula} (3) establishes the correspondence
\[N_{C_2}^{\Ckplusone}(\tee_{h/2^k}^{\Ckplusone}) \left(\frac{a_{\bar{\rho}_{2^{k+1}}}}{a_{\bar{\rho}_{2^{k+1}/2^k}}} \right)^{2^{h/2^k}-1} \leftrightsquigarrow \tee_{h/2^k}^{C_2},\]
where 
\[\frac{a_{\bar{\rho}_{2^{k+1}}}}{a_{\bar{\rho}_{2^{k+1}/2^k}}} = a_{\lambda_k}^{2^{k-1}}a_{\lambda_{k-1}}^{2^{k-2}}\cdots a_{\lambda_1}.\]
Applying the norm, this correspondence becomes 
\[N_{C_2}^{G}(\tee_{h/2^k}^{\Ckplusone})\left(a_{\lambda_n}^{2^{n-1}}a_{\lambda_{n-1}}^{2^{n-2}}\cdots a_{\lambda_{n-k+1}}^{2^{n-k}}\right)^{2^{h/2^k}-1} \leftrightsquigarrow N_{C_2}^{G/\Ck}(\tee_{h/2^k}^{C_2}).\]
Since smashing with $\EF[\Ckplusone]$ inverts the class $a_{\lambda_{n-k}}$, which also inverts the classes $a_{\lambda_n}$, $\ldots$, $a_{\lambda_{n-k+1}}$, the term $\left(a_{\lambda_n}^{2^{n-1}}a_{\lambda_{n-1}}^{2^{n-2}}\cdots a_{\lambda_{n-k+1}}^{2^{n-k}}\right)^{2^{h/2^k}-1}$ in the correspondence above is a unit in $\EF[\Ckplusone] \wedge P_\bullet \left(\BPG\right)$.  The desired equivalence now follows from Equivalence~(\ref{eq:VanishingLineTowerEquivalence}) and the correspondence formula above.  
\end{proof}

To state our theorem about general vanishing lines for $E_h$, recall that from \cref{df:LineLVh-1}, the line $\mathcal{L}_{2^k-1}^V$ on the $(V+t-s, s)$-graded page is defined by the equation 
\[s = (2^k-1)(t-s) + \tau_V(\mathcal{F}_{\leq 2^k}),\]
where by \cref{df:HmaxHmintau}.
\[\tau_V(\mathcal{F}_{\leq 2^k}) = \max_{0 \leq j \leq k} \left(|V^{C_{2^j}}| \cdot 2^j - |V| \right).\]

\begin{df}\rm
For $0 \leq k \leq n$, let $N_k$ be the positive integer 
\[N_k := 2^{h/2^k+n+1} - 2^{n+1}+2^k.\]
Define $\mathcal{L}_{2^k-1, N_k}^V$ to be the line of slope $(2^k-1)$ on the $(V+t-s, s)$-graded page that is given by the equation
\[s = (2^k-1)(t-s) + \tau_V(\mathcal{F}_{\leq 2^k}) + N_k.\]
\end{df}

\begin{thm}[Vanishing Lines]\label{thm:VanishingLineGeneralSlope}
Let $G = \Cnplusone$ and $h = 2^n \cdot m$.  For all $V \in RO(G)$ and $0 \leq k \leq n$, the line $\mathcal{L}_{2^k-1, N_k}^V$ is a vanishing line in the $(V+t-s, s)$-graded page of the slice spectral sequence of $E_h$ within the region $t-s \geq 0$.  This vanishing line satisfies the following properties: 
\begin{enumerate}
\item There are no classes on or above $\mathcal{L}_{2^k-1, N_k}^V$ that survive to the $\mathcal{E}_\infty$-page; 
\item Differentials originating on or above the line $\mathcal{L}_{2^k-1}^V$ have lengths at most ${N_k - (2^k-1) = 2^{h/2^k+n+1} - 2^{n+1} + 1}$; 
\item All differentials originating below $\mathcal{L}_{2^k-1}^V$ have targets strictly below $\mathcal{L}_{2^k-1, N_k}^V$.
\end{enumerate}
\end{thm}

\begin{proof}
When $k = 0$, the line $\mathcal{L}^V_{2^0-1, N_0}$ is the horizontal line 
\[s = 2^{h+n+1}-2^{n+1}+1,\] 
and our claim follows immediately from Theorem~6.1 and Theorem~7.1 in \cite{DuanLiShiVanishing}.

For $1 \leq k \leq n$, we will abbreviate the $G$-spectrum $N_{C_2}^G(\tee^{\Ckplusone}_{h/2^k})^{-1}\BPG$ by $X$ and the $G/\Ck$-spectrum $N_{C_2}^{G/\Ck}(\tee^{C_2}_{h/2^k})^{-1}\BPGmodCk$ by $Y$.  The equivalence of towers 
\[\EF[\Ckplusone] \wedge P_\bullet (X) 
\simeq \Pb_{G/\Ck} \D^k \left(\EF[C_2] \wedge P_\bullet(Y) \right) 
\]
proven in \cref{lem:VanshingLinesLocalizedTowerEquiv}, combined with \cref{thm:ShearingGeneralTheory}, establishes a shearing isomorphism 
\begin{equation} \label{eq:VanishingLineProofShearingIsom1}
\EF[\Ckplusone] \wedge \SliceSS(X) \leftrightsquigarrow \EF[C_2] \wedge \SliceSS(Y).
\end{equation}
This shearing isomorphism establishes a one-to-one correspondence between the $d_{2^kr - (2^k-1)}$-differentials and the $d_r$-differentials.  

On the $C_2$-level of $\EF[C_2] \wedge \SliceSS(Y)$, we have the $C_2$-differential 
\[d_{2^{h/2^k+1}-1}\left(u_{2^{h/2^k}\sigma}\right) = \tee^{C_2}_{h/2^k} a_\sigma^{2^{h/2^k+1}-1}.\]
Since both $\tee^{C_2}_{h/2^k}$ and $a_\sigma$ are inverted, the $C_2$-differential above implies the $C_2$-differential
\[d_{2^{h/2^k+1}-1}\left(u_{2^{h/2^k}\sigma} (\tee^{C_2}_{h/2^k})^{-1} a_\sigma^{1-2^{h/2^k+1}} \right) = 1.\]
Applying the norm $N_{C_2}^{G/\Ck}(-)$ to this differential, we deduce that on the $G/\Ck$-level, the unit class 1 must be killed by a differential of length at most 
\[2^{n-k} \cdot (2^{h/2^k+1}-2) + 1 = 2^{h/2^k + n-k+1} - 2^{n-k+1} + 1\]
(for a discussion of the behavior of norms in the localized slice spectral sequence, see \cite[Section~3.4]{MeierShiZengHF2}).  Under the shearing isomorphism (\ref{eq:VanishingLineProofShearingIsom1}), we deduce that the unit class $1$ in $\EF[\Ckplusone] \wedge \SliceSS(X)$ must be killed by a differential of length at most 
\[2^k\cdot (2^{h/2^k + n-k+1} - 2^{n-k+1} + 1) - (2^k-1) = N_k - (2^k-1). \]

The $G$-equivariant map $X \to E_h$ induces a diagram 
\[\begin{tikzcd}
\SliceSS(X) \ar[r] \ar[d, "\mathcal{L}_{2^k-1}^V"] & \SliceSS(E_h) \ar[d, "\mathcal{L}_{2^k-1}^V"] \\ 
\EF[\Ckplusone] \wedge \SliceSS(X) \ar[r] & \EF[\Ckplusone] \wedge \SliceSS(E_h)
\end{tikzcd}\]
of the corresponding spectral sequences.  By naturality, the unit class 1 in the localized slice spectral sequence ${\EF[\Ckplusone] \wedge \SliceSS(E_h)}$ is killed by a differential of length at most ${N_k - (2^k-1)}$.  This implies that every class in ${\EF[\Ckplusone] \wedge \SliceSS(E_h)}$ must die on or before the $\mathcal{E}_{N_k - (2^k-1)}$-page.  In other words, every class must either support a differential or be hit by a differential of length at most ${N_k - (2^k-1)}$.  

Note that the line $\mathcal{L}^V_{2^{k-1}, N_k}$ by definition is the line $\mathcal{L}^V_{2^{k-1}}$ shifted vertically upwards by $N_k$.  The claim now follows from the Slice Recovery Theorem (\cref{thm:SliceRecovery1}), which states that the map 
\[\SliceSS(E_h) \longrightarrow \EF[\Ckplusone] \wedge \SliceSS(E_h)\]
induces an isomorphism of spectral sequences on or above the line $\mathcal{L}_{2^k-1}^V$ on the $(V+t-s, s)$-graded page within the region $t-s \geq 0$.      
\end{proof}

\begin{figure}
\begin{center}
\makebox[\textwidth]{\includegraphics[trim={0cm 0cm 0cm 0cm}, clip, scale = 0.6]{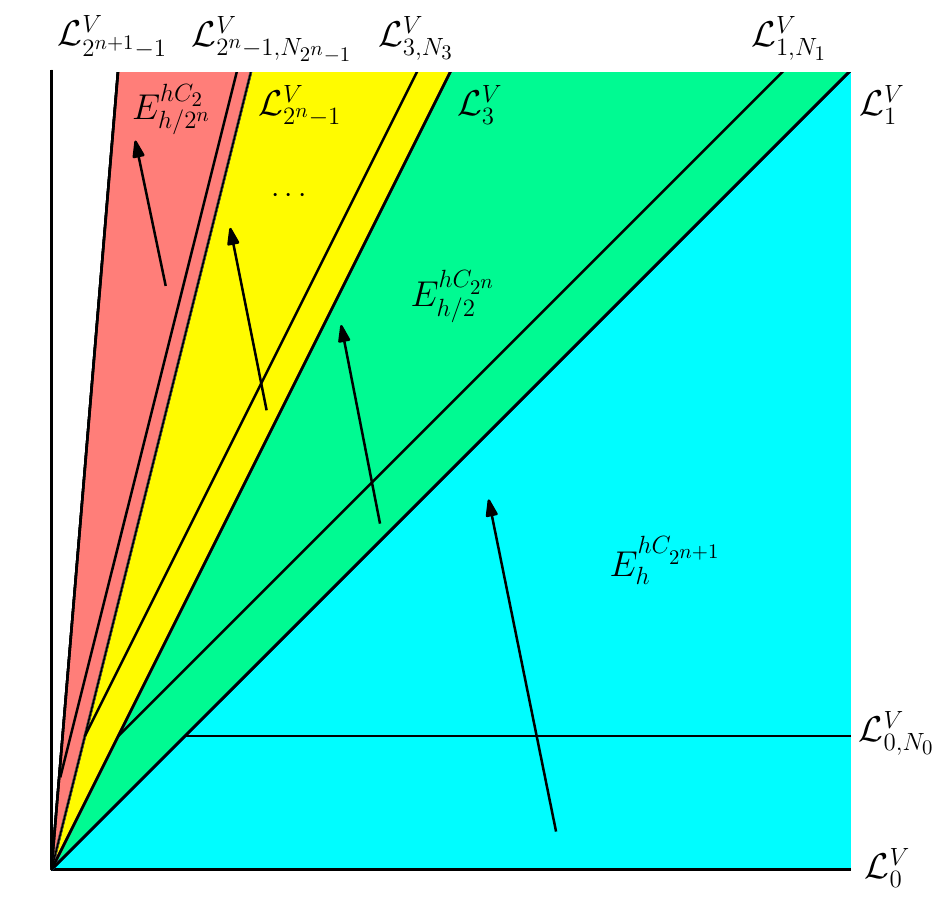}}
\caption{Vanishing lines in the slice spectral sequence of $E_h^{h\Cnplusone}$.}
\hfill
\label{fig:PicVanishingLines}
\end{center}
\end{figure}

\begin{rmk}\rm 
In addition to the stratification provided by the transchromatic tower~(\ref{diagram:TranschromaticTower}), the vanishing lines $\mathcal{L}_{2^k-1, N_k}^V$ (${0 \leq k \leq n}$) established in \cref{thm:VanishingLineGeneralSlope} impose further restrictions on the differentials in the slice spectral sequence of $E_h$ within the region $t-s \geq 0$.  More precisely: 
\begin{enumerate}
\item For $1 \leq k \leq n$, any differential originating below the line $\mathcal{L}_{2^{k}-1}^V$ must have its target strictly below the vanishing line $\mathcal{L}_{2^k-1, N_k}^V$; 
\item Any differential originating on or above the line $\mathcal{L}_{2^n-1}^V$ must have its target on or below the boundary of the positive cone established in \cite[Theorem~C]{MeierShiZengStratification}.  This boundary is given by the line $\mathcal{L}^V_{2^{n+1}-1}$ defined by the equation 
\[s = (|G|-1)(t-s) - |V|+ |G|\cdot \max_{H \subseteq G} |V^H|.\] 
\end{enumerate}
As a consequence, any differential in the slice spectral sequence of $E_h$ must reside within one of the following $(n+1)$ conical regions:
\[\left(\mathcal{L}_0^V, \mathcal{L}_{1, N_1}^V\right), \, \left(\mathcal{L}_1^V, \mathcal{L}_{3, N_3}^V\right), \, \ldots, \, \left(\mathcal{L}_{2^{n-1}-1}^V, \mathcal{L}_{2^{n}-1, N_{2^n-1}}^V\right), \, \left(\mathcal{L}_{2^n-1}^V, \mathcal{L}^V_{2^{n+1}-1} \right).\]  
See \cref{fig:PicVanishingLines}.
\end{rmk}

\begin{rmk}\rm
The arguments presented in the proof of \cref{thm:VanishingLineGeneralSlope} applies to any ${N_{C_2}^G(\tee_{h/2^k}^{\Ckplusone})^{-1}\BPG}$-module $M$ to show that $\mathcal{L}^V_{2^{k}-1, N_k}$ is a vanishing line in the slice spectral sequence of $M$. 
\end{rmk}

\bibliographystyle{alpha}
\bibliography{Bibliography}

\end{document}